\documentclass[pagebackref,colorlinks,citecolor=blue,linkcolor=blue,urlcolor=blue,filecolor=blue]{article}
\pdfoutput=1
\usepackage[margin=1in]{geometry}
\usepackage{tikz-cd}
\usetikzlibrary{calc}
\usepackage{float}
\usepackage{amsmath}
\usepackage{amsfonts}
\usepackage{amsthm}
\usepackage{enumitem}
\usepackage{dsfont}

\usepackage{amssymb}
\usepackage{pifont}
\usepackage{mathtools}
\usepackage{comment}
\usepackage{graphicx, caption} 
\usepackage{float}
\usetikzlibrary{matrix}
\usepackage[T1]{fontenc}
\usepackage[utf8]{inputenc}

\usepackage{pinlabel} 

\usepackage{hyperref}

\newtheorem*{T1}{Theorem~\ref{rep stability for star}}
\newtheorem*{T2}{Theorem~\ref{no new relations}}
\newtheorem*{T3}{Theorem~\ref{infinite presentation for 3}}

\newtheorem{thm}{Theorem}[section]
\newtheorem{lem}[thm]{Lemma}
\newtheorem{prop}[thm]{Proposition}

\newtheorem{cor}[thm]{Corollary}
\newtheorem{exam}[thm]{Example}
\theoremstyle{definition}

\theoremstyle{definition}

\newtheorem{defn}[thm]{Definition}

\newtheorem{remark}[thm]{Remark}

\newtheorem*{claim*}{Claim}
\newtheorem*{quest*}{Question}
\newtheorem*{remark*}{Remark}
\newtheorem*{fact*}{Fact}

\newcommand{\thmtext}{For $k\ge 3$, the sequence of homology groups $H_{i}\big(F_{\bullet}(\Gamma_{k})\big)$ has the structure of a finitely generated FI$_{k,o}$-module.
For $i=1$, this module is generated in degree $4$ for $k=3$, degree $3$ for $k=4$, and degree $2$ for $k\ge 5$; for $i\neq 1$, this module is generated in degree $0$.}

\newcommand{\thmtextone}{For $k\ge 4$, the sequence of homology groups $H_{1}\big(F_{\bullet}(\Gamma_{k})\big)$ has the structure of a finitely presented FI$_{k,o}$-module, presented in degree $6$ for $k=4$, degree $5$ for $k=5$, degree $4$ for $k=6$, and degree $3$ for $k\ge 7$.}

\newcommand{\thmtexttwo}{As an FI$_{3, o}$-module $H_{1}\big(F_{\bullet}(\Gamma_{3})\big)$ is not finitely presentable.}

\newcommand{\Z}{\ensuremath{\mathbb{Z}}}

\newcommand{\R}{\ensuremath{\mathbb{R}}}

\title{Homology generators and relations for the ordered configuration space of a star graph}
\author{Nicholas Wawrykow}
\date{}

\begin{document}
\maketitle
\begin{abstract}
We study the ordered configuration spaces of star graphs.
Inspired by the representation stability results of Church--Ellenberg--Farb for the ordered configuration space of a manifold and the edge stability results of An--Drummond-Cole--Knudsen for the unordered configuration space of a graph, we determine how the ordered configuration space of a star graph with $k$ leaves behaves as we add particles at the leaves.
We show that, as a module over the combinatorial category FI$_{k, o}$, the first homology of this ordered configuration space is finitely generated by $4$ particles for $k=3$, by $3$ particles for $k=4$, and by $2$ particles for $k\ge 5$.
Additionally, we prove that every relation among homology classes can be described by relations on at most $6$ particles for $k=4$, at most $5$ particles when $k=5$, at most $4$ particles when $k=6$, and at most $3$ particles for $k\ge 7$, while proving that adding particles always introduces new relations when $k=3$.
This proves that there is no finite universal presentation for the homology of ordered configuration spaces of graphs.
\end{abstract}

\section{Introduction}

Configuration spaces of particles in manifolds have been studied for nearly a century, though it wasn't until the work of Abrams \cite{abrams2000configuration} and Ghrist \cite{Ghr01} in the late 1990s and early 2000s, that configuration spaces of graphs became a topic of focus for the mathematical community.

Recall that given a space $X$, the \emph{$n^{\text{th}}$-ordered configuration space of particles in $X$} is
\[
F_{n}(X):=\big\{(x_{1}, \dots, x_{n})\in X^{n}|x_{i}\neq x_{j}\mbox{ if }i\neq j\big\},
\]
i.e., the space of ways of embedding $n$ distinct labeled points in $X$.
The \emph{$n^{\text{th}}$-unordered configuration space of particles in $X$}, denoted $C_{n}(X)$,  is the quotient of $F_{n}(X)$ by the natural $S_{n}$-action.

There has been a surge of interest in the topology of configuration spaces in the last decade.
This is especially true in the case $X=\Gamma$ is a graph, that is a $1$-dimensional cellular complex, whose $0$-cells are called \emph{vertices}, and whose $1$-cells are call \emph{edges}.
In particular, there has been considerable effort to understand how graph configurations behave both as the number of particles in the configuration change and as the underlying graphs change.
The papers \cite{an2020edge, an2022asymptotic, aguilar2022farley, an2022second, an2021geometric, chettih2016dancing, chettih2018homology, gonzalez2022cohomology, Lut2017thesis, lutgehetmann2017representation, ramos2020application} are among the many works in these directions.
We continue these efforts, seeking to determine to how the homology of the ordered configuration space of a \emph{star graph} behaves as the number of particles in the configuration increases.

\subsection{Stability in the Particle Direction}
A common technique for determining the structure of the configuration space of particles in $X$ is to study 
how the configuration space of $X$ behaves as the number of particles in the configuration increases.
If $X$ is a connected non-compact manifold of dimension at least $2$, McDuff \cite{mcduff1975configuration} and Segal \cite{segal1979topology} noted that if $n$ is sufficiently large with respect to $i$, there is an isomorphism
\[
H_{i}\big(C_{n}(X)\big)\xrightarrow{\sim} H_{i}\big(C_{n+1}(X)\big).
\]
This \emph{homological stability} arises from stabilization maps
\[
\iota:C_{n}(X)\hookrightarrow C_{n+1}(X)
\]
that ``add a particle at infinity'' to the configuration.

In the case of the unordered configuration space of a graph, the analogous operation of adding a particle at a valence $1$ vertex, i.e., a \emph{leaf}, induces similar maps on configuration space, though such maps rarely yield homological stability results.
To ameliorate this, An--Drummond-Cole--Knudsen \cite{an2020edge} introduced \emph{edge stabilization} maps that add a new particle at the interior of an edge.
These maps induce a $\Z[E]$-module structure on $H_{i}\big(C_{\bullet}(\Gamma)\big)$, where $E$ is the set of edges of $\Gamma$.
They proved that $H_{i}\big(C_{\bullet}(\Gamma)\big)$ is finitely presented as a $\Z[E]$-module \cite[Theorem 1.1]{an2020edge}, and they proved that the rank of $H_{i}\big(C_{n}(\Gamma)\big)$ is a polynomial in $n$ of degree determined by a connectivity invariant of $\Gamma$ \cite[Theorem 1.2]{an2020edge}.

In the case of ordered configuration spaces, the labeling of the particles often precludes a notion of homological stability.
However, this labeling induces an $S_{n}$-representation structure on $H_{i}\big(F_{n}(X)\big)$.
Church--Ellenberg--Farb proved that if one takes this structure into account and $X$ is a connected non-compact finite type manifold of dimension at least $2$, then $\Big\{H_{i}\big(F_{n}(X)\big)\Big\}_{n\ge 0}$ stabilizes as a sequence of symmetric group representations, as it forms a finitely generated free module over the combinatorial category FI \cite[Theorem 6.4.3]{church2015fi}.
Moreover, this \emph{representation stability} also arises from the map on configuration space that ``adds a particle at infinity.''

Inspired by the work of An--Drummond-Cole--Knudsen and Church--Ellenberg--Farb, we approach the problem of representation stability for the family of star graphs.
The \emph{star graph} $\Gamma_{k}$ consists of a central vertex with $k$ surrounding vertices, each connected by a single edge to the central vertex; see Figure \ref{star3} for an example.
In the case of the unordered configuration space of an arbitrary graph $\Gamma$, An--Drummond-Cole--Knudsen proved that as $n$ tends to infinity the vast majority of all homology classes of $F_{n}(\Gamma)$ arise from the inclusions of star graphs into $\Gamma$, i.e., most classes in $H_{i}\big(C_{n}(\Gamma)\big)$ are products of star classes \cite[Theorem 1.1]{an2022asymptotic}.
An argument analogous to the one given by An--Drummond-Cole--Knudsen in \cite{an2020edge} proves that for each $n$, the $k$ maps that push the particles in a configuration lying on the $j^{\text{th}}$-edge towards the central vertex and add a particle labeled $n+1$ at the $j^{\text{th}}$-leaf 
\[
\iota_{n,j}:F_{n}(\Gamma_{k})\hookrightarrow F_{n+1}(\Gamma_{k})
\]
induce maps on homology that give $H_{i}\big(F_{\bullet}(\Gamma_{k})\big)$ the structure of a module over the combinatorial category FI$_{k,o}$.
We prove that $H_{i}\big(F_{\bullet}(\Gamma_{k})\big)$ is finitely generated as an FI$_{k,o}$-module, a relative of the FI$_{d}$-modules studied by Ramos in \cite{ramos2017generalized}, and calculate its generation degree as such.

\begin{thm}\label{rep stability for star}
\thmtext
\end{thm}

Notably, this representation stability result proves that for $n=3,4$, there are classes in $H_{1}\big(F_{n}(\Gamma_{3})\big)$ that cannot be viewed as a sum of a product of a class of $H_{1}\big(F_{n-1}(\Gamma_{3})\big)$ with a copy of the fundamental class of $H_{0}\big(F_{1}(\R)\big)$.
This differs from the case of the first homology of the unordered configuration space of $\Gamma_{3}$, which An--Drummond-Cole--Knudsen's results prove is freely generated by the fundamental class of $H_{1}\big(C_{2}(\Gamma_{3})\big)$ and the $\Z[E]$-module structure on $H_{1}\big(C_{\bullet}(\Gamma_{3})\big)$.
Thus, we see that ordered configuration spaces of graphs are more complex than their unordered counterparts, even in the simplest case.

Our proof of Theorem \ref{rep stability for star} also provides an upper bound of the presentation degree of $H_{1}\big(C_{\bullet}(\Gamma_{k})\big)$ as an FI$_{k,o}$-module.

\begin{thm}\label{no new relations}
\thmtextone
\end{thm}

This finite presentability result for the ordered configuration space of a star graph on at least $4$ edges is in line with An--Drummond-Cole--Knudsen's edge stability results for the unordered configuration space of a graph \cite[Theorem 1.1]{an2020edge} and Church--Ellenberg-Farb's representation stability results for the ordered configuration space of a connected non-compact finite type manifold of dimension at least 2 \cite[Theorem 6.4.3]{church2015fi}.
We prove that the case $k=3$ is special, deviating from this trend.
Namely, we show that $H_{1}\big(F_{\bullet}(\Gamma_{3})\big)$ is not finitely presentable as an FI$_{3,o}$-module.

\begin{thm}\label{infinite presentation for 3}
\thmtexttwo
\end{thm}

This is in stark contrast to previous results, suggesting that the limited $1$-dimensional cell structure of $\Gamma_{3}$ greatly increases the complexity of its ordered configuration space.
This proves that there is no finite universal presentation for the homology of the ordered configuration space of graphs, an open problem in the unordered case; see \cite[Remark 3.14]{an2020edge} for more on this problem.

\begin{figure}[h]
\centering
\captionsetup{width=.8\linewidth}
\includegraphics[width = 2cm]{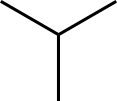}
\caption{The star graph $\Gamma_{3}$.
}
\label{star3}
\end{figure}

\subsection{Structure of the Paper}
In section \ref{FIOk} we recall the categories FI and FI-mod, as well as the representation stability results of Church--Ellenberg--Farb for ordered configuration spaces of points in manifolds.
We introduce the category FI$_{d,o}$, an ordered version of the category FI$_{d}$, that we will use to formalize our representation stability results.
Next, in section \ref{cube complex}, we recall the definition of a cubical complex defined by L\"{u}tgehetmann that is homotopy equivalent to the ordered configuration space of a graph.
After that, we recall the Mayer--Vietoris spectral sequence in section \ref{mayer vietoris}, which, along with the cubical complex of section \ref{cube complex}, will be central to our generation and presentation arguments.
In section \ref{rep stab section}, we consider generation of $H_{1}\big(F_{\bullet}(\Gamma_{k})\big)$ as an FI$_{k, o}$-module, and we prove Theorem \ref{rep stability for star} by finding a ``good'' cover for L\"{u}tgehetmann's cube complex and analyzing the first row of the Mayer--Vietoris spectral sequence with respect to this cover.
Next, in section \ref{presentability} we prove that if $k\ge 4$, then $H_{1}\big(F_{\bullet}(\Gamma_{k})\big)$ is a finitely presentable FI$_{k,o}$-module, i.e., Theorem \ref{no new relations}, whereas $H_{1}\big(F_{\bullet}(\Gamma_{3})\big)$ is not finitely presentable as an FI$_{3, o}$-module, i.e., Theorem \ref{infinite presentation for 3}.
Finally, the paper concludes with Section \ref{appendix}, a short appendix in which we illustrate the homology generators of $H_{1}\big(F_{\bullet}(\Gamma_{k})\big)$.\\

\subsection{Acknowledgements}
The author would like to thank Shmuel Weinberger for posing a question that lead the author to consider graph configuration spaces.
The author would like to thank Benson Farb, Ben Knudsen, Andrew Snowden, and Jennifer Wilson for insightful conversations, as well as Jes\'{u}s Gonz\'{a}lez and Tomi Rossini for comments on a draft of this paper. 
Additionally, the comments of an anonymous referee were of great help in improving the clarity and correctness of the paper.

\section{FI$_{d,o}$}\label{FIOk}

Church--Ellenberg--Farb used the combinatorial category FI to state their representation stability results for connected non-compact finite type manifolds of dimension at least $2$. 
In this section we recall the definitions of the category FI and of FI-modules, and introduce generalizations, which we will call FI$_{d,o}$ and FI$_{d,o}$-modules, that we will use to state our representation stability results.

Let \emph{FB} denote the category whose objects are finite sets and whose morphisms are bijections.
An \emph{FB-module} over a commutative ring $R$ is a covariant functor from FB to the category of $R$-modules, and an \emph{FB-space} is a covariant functor from FB to the category of spaces.
Since every finite set $A$ is isomorphic to a set of the form $[n]:=\{1, \dots, n\}$ for some $n\in \Z_{\ge 0}$, it follows that FB is equivalent to its skeleton, which has one object $[n]$ for each $n\in \Z_{\ge 0}$.
As such, we can interpret an FB-module $W$ as a sequence $(W_{n})_{n\in \Z_{\ge0}}$ of symmetric group representations.

If we allow injections between sets of different cardinalities in FB, we get the category \emph{FI} of finite sets and injections.
An \emph{FI-module} over a commutative ring $R$ is a covariant functor from FI to the category of $R$-modules, and an \emph{FI-space} is a covariant functor from FI to the category of topological spaces.
As is the case for FB, the category FI is equivalent to its skeleton, which has one object $[n]$ for each $n\in \Z_{\ge 0}$.

\begin{exam}
If $X$ is a non-compact finite type manifold of dimension at least $2$, then $F_{\bullet}(X)$ is an FI-space, where the standard inclusion $\iota_{n}:[n]\hookrightarrow[n+1]$ induces a map that sends $F_{n}(X)$ to $F_{n+1}(X)$ by ``adding a particle labeled $n+1$ at infinity''; see Figure \ref{torusconfiginclusion}.
It follows that $H_{i}\big(F_{\bullet}(X)\big)$ has the structure of an FI-module where the inclusion map $\iota_{n}$ induces a map from $H_{i}\big(F_{n}(X)\big)$ to $H_{i}\big(F_{n+1}(X)\big)$ that tensors a class in $H_{i}\big(F_{n}(X)\big)$ with the fundamental class of $H_{0}\big(F_{1}(\R^{n})\big)$.
\end{exam}

\begin{figure}[h]
\centering
\captionsetup{width=.8\linewidth}
\includegraphics[width = 12cm]{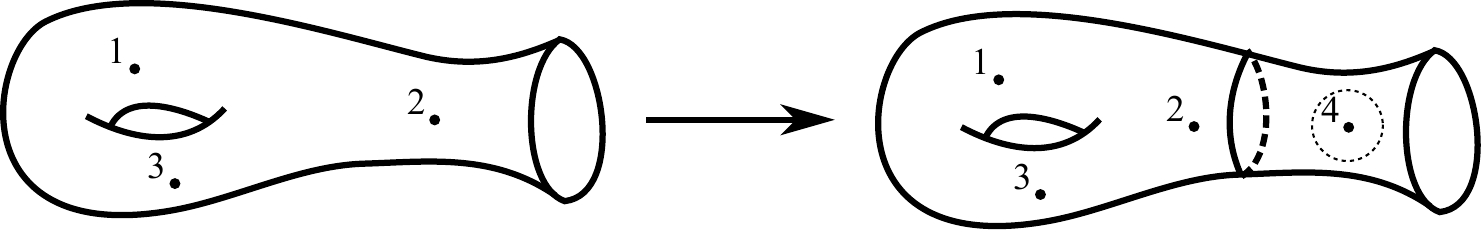}
\caption{The action of the inclusion map on a point in $F_{3}(T^{\circ})$, where $T^{\circ}$ is the once-punctured torus.
This map is induced by a map on the torus that retracts the torus away from the puncture; note that in the coimage of this retraction, which is a tubular neighborhood of the puncture, one can embed a copy of $\R^{2}$ (see the dotted circle).
To get a point in $F_{4}(T^{\circ})$ from a point in $F_{3}(T^{\circ})$ first apply this retraction to the three particles constituting this point; then add a new particle centered at the image of the origin of the embedded $\R^{2}$.}
\label{torusconfiginclusion}
\end{figure}

We say that an FI-module $V$ is \emph{generated} by a set $S\subseteq \sqcup_{n\ge 0}V_{n}$ if $V$ is the smallest FI-submodule containing $S$. 
If there is some finite set $S$ that generates $V$, then $V$ is \emph{finitely generated}, and if $V$ is generated by $\bigsqcup_{0\le n\le m}V_{n}$, then $V$ is said to be \emph{generated in degree at most $m$}.

Church--Ellenberg--Farb proved that the homology groups of the ordered configuration space of a connected non-compact manifold of dimension at least $2$ stabilize in a representation theoretic sense as they are finitely generated FI-modules.

\begin{thm}\label{manifold rep stab}
(Church--Ellenberg--Farb \cite[Theorem 6.4.3]{church2015fi} for oriented manifolds, Miller--Wilson \cite[1.1]{miller2019higher} for the general case)
Let $M$ be a connected non-compact finite type manifold of dimension at least $2$, then $H_{i}\big(F_{\bullet}(M)\big)$ has the structure of a finitely generated free FI-module generated in degree at most $2i$.
\end{thm}

The homology of the ordered configuration space of particles on a graph rarely has the structure of a finitely generated FI-module.
Unlike the ordered configuration space of a manifold of dimension at least $2$, the order in which we add particles ``at infinity'' to the ordered configuration space of a graph matters.
Moreover, for graphs, the ``infinity,'' i.e., leaf, of a graph to which one adds new particles matters: One can check that all the classes arising from adding particles on different edges to the star class on two particles described in subsection \ref{star class on 2 particles} are linearly independent, whereas the analogous statement for manifolds of dimension at least $2$ is false.
Therefore, we need a new combinatorial category that takes into account the order in which we add particles and the various places we can add them.

For fixed $d\ge 1$, let $\emph{FI}_{d}$ denote the category whose objects are finite sets and whose morphisms $(f, c)$ are injections $f$ with a $d$-coloring $c$ of the complement of the image of $f$.
This category, which was studied by Ramos in \cite{ramos2017generalized}, allows one to take into account the different ways (edges) in which we add a new particle; however, it does not remember the order in which we added them.
To include this information, we introduce a new combinatorial category.

\begin{defn}
Let \emph{FI$_{d,o}$} denote the category whose objects are finite sets and whose morphisms $(f,c,o)$ are injections $f$ with a $d$-coloring $c$ of the complement of the image of $f$ and an ordering $o$ for each set of elements of the same color.
We can compose FI$_{d,o}$-morphisms in the following way: if the codomain of an injection $f$ is the domain of an injection $f'$, we compose $(f,c,o):[n]\to [m]$ and $(f', c', o'):[m]\to [l]$ by setting
\[
(f', c', o')\circ (f, c, o)=(f'\circ f, c'', o''),
\]
where for $x\notin\text{Im}(f'\circ f)$
\[
c''(x)=
\begin{cases} c(x)&\mbox{if }x\in \text{Im}f'\\
c'(x)&\mbox{if }x\notin \text{Im}f'
\end{cases}\indent\text{and}\indent
o''(x)=\begin{cases}o(x)&\mbox{if }x\in \text{Im}f'\\
o'(x)+(n-m)&\mbox{if }x\notin \text{Im}f'.
\end{cases}
\]
See Figure \ref{compFIO2morph} for an example.

Additionally, an \emph{FI$_{d,o}$-module} over a ring $R$ is a covariant functor from FI$_{d,o}$ to $R$-mod, and an \emph{FI$_{d,o}$-space} is a covariant function from FI$_{d,o}$ to the category of topological spaces.
\end{defn}

\begin{figure}[h]
\centering
\captionsetup{width=.8\linewidth}
\includegraphics[width = 8cm]{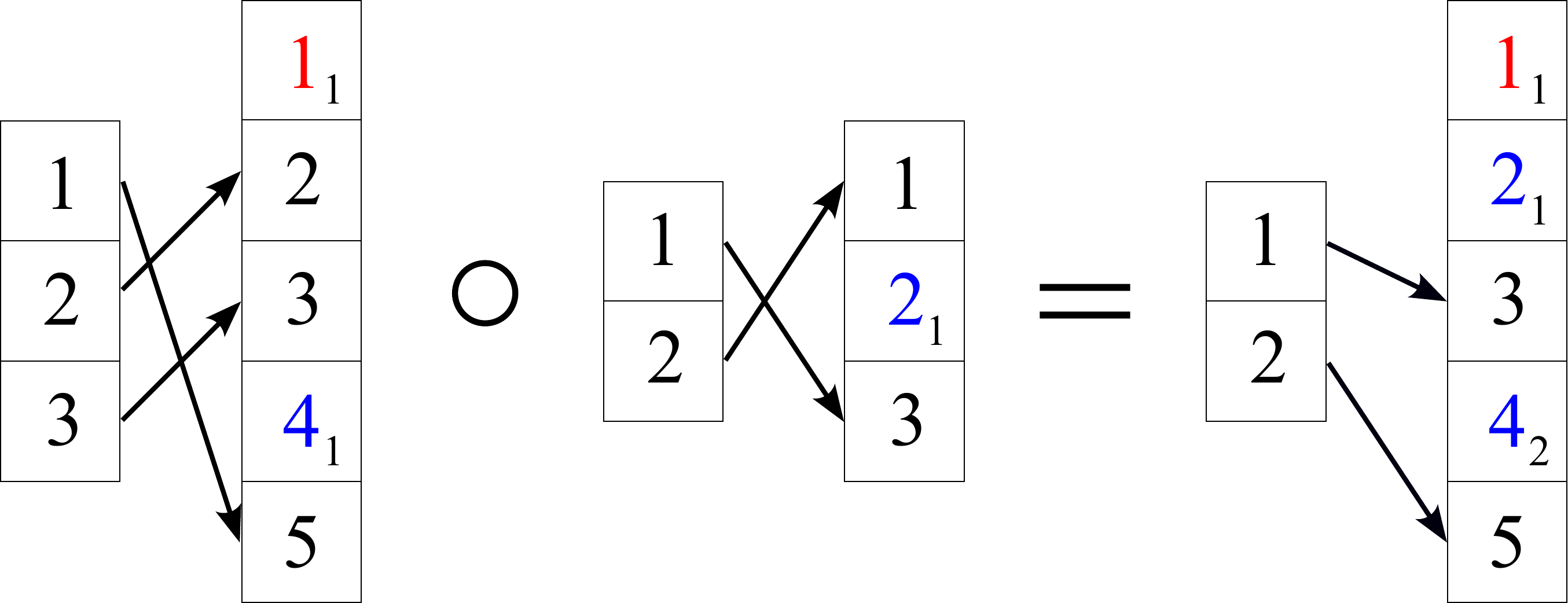}
\caption{The composition of two FI$_{2,o}$-morphisms.
Here the ordering is denoted by the subscripts.
}
\label{compFIO2morph}
\end{figure}

Note that every morphism $(f, c, o):[n]\to [m]$ can be written
\[
(f, c, o)=\sigma\circ (\iota_{m-1}, c_{j_{m}}, o_{j_{m}})\circ\cdots\circ (\iota_{n}, c_{j_{n+1}}, o_{j_{n+1}}),
\]
where $\sigma\in S_{m}$, $\iota_{l}:[l]\to [l+1]$ is the standard inclusion of $[l]$ into $[l+1]$, $c_{j_{l+1}}$ colors the element $l+1$ color $j_{l+1}$, and $o_{j_{l+1}}$ orders the element $l+1$ first among the elements colored $j_{l+1}$ as it is the only element in the complement; see Figure \ref{FIO2decompmorph}.

\begin{figure}[h]
\centering
\captionsetup{width=.8\linewidth}
\includegraphics[width = 12cm]{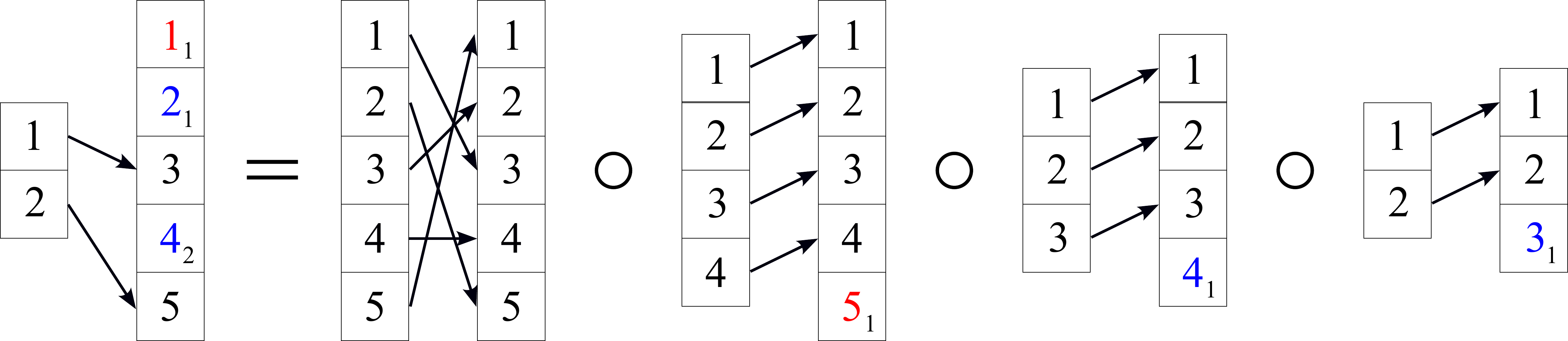}
\caption{A decomposition of the FI$_{2,o}$-morphism depicted on the right of Figure \ref{compFIO2morph} into FI$_{2,o}$-morphisms of the form $\sigma\in S_{5}$ and $(\iota_{l-1}, c_{j_{l}}, o_{j_{l}})$. 
}
\label{FIO2decompmorph}
\end{figure}

Like FB and FI, the category FI$_{d, o}$ is equivalent to its skeleton, which has one object $[n]$ for each $n\in \Z_{\ge0}$, and we will think of FI$_{d,o}$ as this skeleton.
Given an FI$_{d,o}$-module $U$, we write $U_{n}$ for $U\big([n]\big)$. 
The set of FI$_{d,o}$-morphisms from $[n]$ to itself can be identified with the symmetric group $S_{n}$, making $U_{n}$ an $S_{n}$-representation.

Given $U$, an \emph{FI$_{d,o}$-submodule} $U'$ of $U$ is a sequence of symmetric group representations $U'_{n}\subseteq U_{n}$ that is closed under the induced actions of the FI$_{d,o}$-morphisms.
We say that $U$ is \emph{generated} by a set $S\subseteq \bigsqcup_{n\ge 0}U_{n}$ if $U$ is the smallest FI$_{d,o}$-submodule of itself containing $S$. 
If there is some finite set $S$ that generates $U$, then we say that $U$ is \emph{finitely generated}.
If $U$ is generated by $S\subseteq \bigsqcup_{0\le n\le m}U_{n}$, we say that $U$ is \emph{generated in degree at most $m$}.

Next, we provide a pair of necessary and sufficient conditions for an FB-space $X_{\bullet}$ with $d$ maps $\iota_{n, j}:X_{n}\to X_{n+1}$ for all $n$ to be an FI$_{d,o}$-space.
We will use this to prove that $H_{i}\big(F_{\bullet}(\Gamma_{k})\big)$ has the structure of an FI$_{k,o}$-module.

Given an FB-space $X$ with $d$ \emph{insertion maps} $\iota_{n, j}:X_{n}\to X_{n+1}$ for every $n$, we say that \emph{insertion maps commute with permutations} if for every $j$ and $n$ and every $\sigma\in S_{n}$, we have
\[
\iota_{n, j}\circ\sigma=\tilde{\sigma}\circ \iota_{n, j}
\]
where $\tilde{\sigma}\in S_{n+1}$ is the image in $S_{n+1}$ of $\sigma$ under the standard inclusion $\iota_{n}:[n]\hookrightarrow[n+1]$.
Additionally, if for every $j\neq l$ and for every $n$, we have that 
\[
(n+1\, n+2)\circ \iota_{n+1, l}\circ \iota_{n, j}=\iota_{n+1, j}\circ \iota_{n, l},
\]
then we say that \emph{distinct insertions are unordered}.

\begin{prop}\label{conditions for FIOk space}
Let $X$ be an FB-space such that for all $n$ there exist $d$ insertion maps $\iota_{n, j}:X_{n}\to X_{n+1}$.
If these insertion maps commute with permutations and distinct insertions are unordered, then $X$, along with these insertion maps, has the structure of an FI$_{d,o}$-space.
\end{prop}

\begin{proof}
For each morphism $(f, c, o):[n]\to [m]$ in FI$_{d,o}$, we need to define a map $(f, c, o)_{*}:X_{n}\to X_{m}$.
When $n=m$, we have $f=\sigma$ for some $\sigma\in S_{n}$, and $c=\cdot$ and $o=\cdot$ are the trivial coloring and ordering, respectively.
In this case $(f, c, o)_{*}=\sigma$, where we also use $\sigma$ to denote its action on $X_{n}$.
The insertion maps $\iota_{l,j_{l+1}}$ describe what happens when $f$ is the standard inclusion $\iota_{l}:[l]\hookrightarrow [l+1]$, $c_{j_{l+1}}$ colors the element $l+1$ color $j_{l+1}$, and $o_{j_{l+1}}$ orders the element $l+1$ the first among the elements colored $j=j_{l+1}$, i.e., $(f, c, o)=(i_{l}, c_{j_{l+1}}, o_{j_{l+1}})$ and $(f, c, o)_{*}=\iota_{l, j_{l+1}}$.

Given an arbitrary morphism $(f, c, o):[n]\to[m]$ in FI$_{d,o}$ we can write $f:[n]\to [m]$ as $\sigma'\circ \iota_{m-1, j_{m}}\circ\cdots\circ \iota_{n, j_{n+1}}$ for $\sigma'\in S_{m}$.
Moreover, there is a unique such $\sigma'$ that sends $1$ to $f(1)$, $2$ to $f(2)$, etc., while sending $n+1$ to the first element in complement of the image of $f$ colored $j_{n+1}$, $n+2$ to the first element colored $j_{n+2}$ complement of the image $f$ if $j_{n+2}\neq j_{n+1}$ and the second element colored $j_{n+1}$ if $j_{n+2}=j_{n+1}$, and so on and so forth.
It follows that 
\[
(f, c, o)=(\sigma', \cdot, \cdot)\circ(\iota_{m-1}, c_{j_{m}}, o_{j_{m}})\circ\cdots\circ (\iota_{n}, c_{j_{n+1}}, o_{j_{n+1}}),
\]
so we should define
\[
(f, c, o)_{*}:=\sigma'\circ\iota_{m-1, j_{m}}\circ\cdots\circ\iota_{n, j_{n+1}}.
\]

To check functoriality, we must show that if we have another sequence of permutations and insertions that compose to $(f, c, o)$ in FI$_{d,o}$, then the corresponding maps on the various spaces $X_{l}$ compose to $(f, c, o)_{*}$.
Given an arbitrary sequence of permutations and insertion maps, the property that insertion maps commute with permutations implies that we can push all the permutations to the left of the insertion maps without changing the composition map, obtaining a composition of insertions followed by a composition of permutations.
Since there is an $S_{n}$-action on $X_{n}$, we can replace the composition of permutations by a single permutation.
Therefore, it suffices to show that if
\[
(f, c, o)=(\sigma'', \cdot, \cdot)\circ (\iota_{m-1}, c_{j'_{m}}, o_{j'_{m}})\circ \cdots\circ(\iota_{n}, c_{j'_{n+1}}, o_{j'_{n+1}}),
\]
then
\[
\sigma''\circ\iota_{m-1, j'_{m}}\circ \cdots\circ\iota_{n, j'_{n+1}}=\sigma'\circ\iota_{m-1, j_{m}}\circ \cdots\circ\iota_{n, j_{n+1}}.
\]

We proceed by induction on $m-n$, noting that the above proves that this is true if $m-n=0$ or $1$.
Note that $\sigma'$ and $\sigma''$ take the same values on $[n]$ as $f$, so we can write $\sigma''=\sigma'\circ \tilde{\sigma}$ where $\tilde{\sigma}$ only permutes $[m]\backslash[n]$.
Therefore, after canceling $\sigma'$ from both sides, it suffices to show that 
\[
\tilde{\sigma}\circ\iota_{m-1, j'_{m}}\circ\cdots\circ\iota_{n, j'_{n+1}}=\iota_{m-1, j_{m}}\circ \cdots\circ\iota_{n, j_{n+1}}.
\]
This follows from the fact that distinct insertions are unordered.
We see this by induction on $m-n$.
If $m-n=1$, this is trivial.
Otherwise, let $n+k=(\tilde{\sigma})^{-1}(n+1)$, i.e., in the alternate composition, element $n+k$ gets inserted with color $j'_{n+k}=j_{n+1}$ and then $\tilde{\sigma}$ changes it to $n+1$.
Since distinct insertions are unordered if $j'_{n+k}\neq j'_{n+k-1}$, we can write
\[
\iota_{n+k-1, j'_{n+k}}\circ\iota_{n+k-2, j'_{n+k-1}}=(n+k\;\;\;n+k-1)\circ \iota_{n+k-1, j'_{n+k-1}}\circ\iota_{n+k-2, j'_{n+k}}.
\]
Moreover, we need not worry about the case $j'_{n+k}=j'_{n+k-1}$ as this would imply that the two compositions have a different first element of color $j'_{n+k}$ in the ordering given by $o$, a contradiction.
By applying the fact that insertion maps commute with permutations, we can move all the transpositions to the left to make the composition
\[
\tilde{\sigma}\circ (n+k\;\;\; n+k-1)\circ\cdots\circ (n+2\;\;\; n+1).
\]
This is equal to 
\[
\tilde{\sigma}\circ(n+k\;\;\;n+k-1\;\;\;\cdots\;\;\; n+2\;\;\; n+1),
\]
which fixes $n+1$.
Denote this composition by $\tilde{\sigma}'$.
It follows that
\[
\tilde{\sigma}\circ\iota_{m-1, j'_{m}}\circ\cdots\circ\iota_{n, j'_{n+1}}=\tilde{\sigma}'\circ \iota_{m-1, j'_{m}}\circ\cdots\circ\iota_{n+k, j'_{n+k+1}}\circ \iota_{n+k-1, j'_{n+k-1}}\circ\cdots\circ\iota_{n+1, j'_{n+1}}\circ\iota_{n, j'_{n+k}},
\]
and $j'_{n+k}=j_{n+1}$.
By the induction hypothesis, we have that 
\[
\tilde{\sigma}'\circ \iota_{m-1, j'_{m}}\circ\cdots\circ\iota_{n+k, j'_{n+k+1}}\circ\iota_{n+k-1, j'_{n+k-1}}\circ\cdots\circ\iota_{n+1, j'_{n+1}}
\]
and
\[
\iota_{m-1, j_{m}}\circ\cdots\circ\iota_{n+1, j_{n+2}}
\]
induce the same maps from $[n+1]$ to $[m]$ as they are both induced by $(f, c, o)$---here $n+1$ is not colored or ordered, rather it is sent to the element that is the first colored $j_{n+1}$ by  $(f, c, o)$---so by composing both of these on the right with $\iota_{n, j'_{n+k}}=\iota_{n, j_{n+1}}$ gives the desired equality, completing the inductive step and the proof of the proposition.
\end{proof}

Note that there exist $k$ insertion maps $\iota_{n, j}:F_{n}(\Gamma_{k})\to F_{n+1}(\Gamma_{k})$ that give $F_{\bullet}(\Gamma_{k})$ the structure of an FI$_{k,o}$-space.

\begin{lem}\label{Fstar is an FIOk space}
Let $s_{j}:\Gamma_{k}\to \Gamma_{k}$ be the continuous map that scales edge $j$ by a factor of $\frac{4}{5}$ and is the identity on the other edges, and let 
\[
\iota_{n, j}: F_{n}(\Gamma_{k}) \to F_{n+1}(\Gamma_{k})
\]
be the map induced by $s_{j}$ that adds a particle labeled $n+1$ at the end of the $j^{\text{th}}$-leaf of $\Gamma_{k}$.
See Figure \ref{inclusiongraphconf}.
Then, for all $k\ge 3$, the ordered configuration space of particles in the star graph with $k$ leaves, along with these insertion maps, has the structure of an FI$_{k,o}$-space.
\end{lem}

\begin{proof}
By Proposition \ref{conditions for FIOk space} it suffices to check that our insertion maps $\iota_{n,j}$ are unordered and commute with permutations.
This follows immediately from the definition of the $\iota_{n,j}$.
\end{proof}

\begin{figure}[h]
\centering
\captionsetup{width=.8\linewidth}
\includegraphics[width = 6cm]{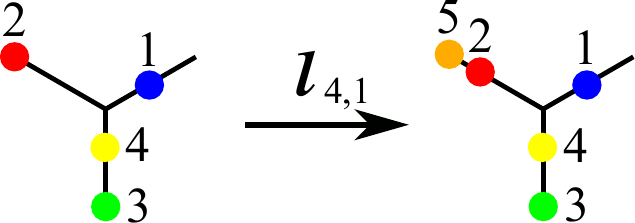}
\caption{The action of $\iota_{4,1}$ on a point in $F_{4}(\Gamma_{3})$. 
From now on we will suppress the naming of the particles by natural numbers in our figures, and instead use colors to identify them.
These are not the colors associated to the FI$_{k,o}$-modules, which are the edges of $\Gamma_{k}$.
}
\label{inclusiongraphconf}
\end{figure}

Since homology with $R$-coefficients is a covariant functor from the category of topological spaces to the category of $R$-modules, the following corollary is immediate.

\begin{cor}\label{homology is an FIOk mod}
For all $k\ge 3$ and all $i$, the sequence of homology groups $H_{i}\big(F_{\bullet}(\Gamma_{k})\big)$ has the structure of an FI$_{k,o}$-module.
\end{cor}

This structure comes from noting that $\iota_{n,j}$ induces a map
\[
(\iota_{n, j})_{*}:H_{i}\big(F_{n}(\Gamma_{k})\big)\to H_{i}\big(F_{n+1}(\Gamma_{k})\big)
\]
that tensors a class in $H_{i}\big(F_{n}(\Gamma_{k})\big)$ with the fundamental class of $H_{0}\big(F_{1}(\R)\big)$ along the $j^{\text{th}}$-edge. 

To determine the generation and presentation degrees of our FI$_{d,o}$-modules we introduce the notion of free FI$_{d,o}$-modules.
The \emph{free} FI$_{d,o}$-module generated in degree $n$ is
\[
M^{\text{FI}_{d,o}}(n)_{\bullet}:=\text{Hom}_{\text{FI}_{d,o}}\big([n], \bullet\big),
\]
i.e., $M^{\text{FI}_{d,o}}(n)_{m}:=\text{Hom}_{\text{FI}_{d,o}}\big([n], [m]\big)$ and the FI$_{d,o}$-module morphisms from $M^{\text{FI}_{d,o}}(n)_{m}$ to $M^{\text{FI}_{d,o}}(n)_{l}$ arise from the FI$_{d,o}$-morphisms from $[m]$ to $[l]$.
If $W_{n}$ is an $S_{n}$-representation, we set
\[
M^{\text{FI}_{d,o}}(W_{n})_{\bullet}=W_{n}\otimes_{S_{n}}M^{\text{FI}_{d,o}}(n)_{\bullet}.
\]
We extend this definition to FB-modules by setting
\[
M^{\text{FI}_{d,o}}(W)_{\bullet}=\bigoplus_{n=0}M^{\text{FI}_{d,o}}(W_{n})_{\bullet}.
\]

Free $\text{FI}_{d,o}$-modules allow us to reformulate the notion of finite-generation of an $\text{FI}_{d,o}$-module.
Namely, an $\text{FI}_{d,o}$-module $U$ is finitely generated if and only if there is a surjection
\[
M^{\text{FI}_{d,o}}(W)\twoheadrightarrow U,
\]
where $W$ is a finitely generated FB-module, i.e., an FB-module such that $W_{n}=0$ for all but finitely many $n$, and if $W_{n}\neq 0$, then $W_{n}$ is a finite dimensional $S_{n}$-representation.
This also allows us to define finitely presented $\text{FI}_{d,o}$-modules.
We say that $U$ is a \emph{finitely presented} $\text{FI}_{d,o}$-module if there exist finitely generated FB-modules $W$ and $V$ such that the sequence
\[
M^{\text{FI}_{d,o}}(V)\to M^{\text{FI}_{d,o}}(W)\to U\to 0
\]
is exact.
The \emph{presentation degree} of $W$ is the minimum of the generation degrees of $V$ and $W$ as FB-modules.

\subsection{Dimension and non-Noetherianity}
While we will not need the following information about FI$_{d,o}$-modules in the rest of the paper, we include it as it is of independent interest.

Note that there are $(m)\cdots(m-n+1)$ injections $\phi:[n]\hookrightarrow[m]$.
For the remaining $m-n$ elements of the image of $[m]\backslash\phi\big([n]\big)$ and each weak $d$-composition $(a_{1}, \dots, a_{d})$ of $m-n$ there are $\binom{m-n}{a_{1}!\cdots a_{d}!}$ ways of coloring $a_{1}$ of the elements color $1$, coloring $a_{2}$ of the elements $2$, so on and so forth.
There are $a_{1}!$ ways of ordering the color $1$ elements, $a_{2}!$ ways of ordering the color $2$ elements, etc.
Since there are $\binom{m-n+d-1}{d-1}$ weak $d$-compositions of $m-n$, we get the following proposition on the dimension of free FI$_{d, o}$-modules.

\begin{prop}\label{size of a free FIOk}
For $m\ge n$
\[
\dim\big(M^{\text{FI}_{d,o}}(n)_{m}\big)=m!\binom{m-n+d-1}{d-1}
\]
\end{prop}

In fact, more can be said: As $S_{m}$-representations
\[
M^{\text{FI}_{d,o}}(W_{n})_{m}=\bigoplus_{(a_{1}, \dots, a_{d})}\big(\text{Ind}^{S_{m}}_{S_{n}\times S_{a_{1}}\times\cdots\times S_{a_{d}}}W_{n}\boxtimes \text{Reg}_{a_{1}}\boxtimes\cdots\boxtimes \text{Reg}_{a_{d}}\big),
\]
where $\text{Reg}_{a_{i}}$ is the regular representation of $S_{a_{i}}$.

\begin{remark}
The results in this section could be restated in the language of twisted algebras.
In particular, an FI$_{d,o}$-module can be thought of as module over the twisted algebra version of the $d$-fold tensor product of the tensor algebra on a $1$-dimensional vector space, i.e., the twisted algebra generated by the tensor product of $d$ copies of the twisted tensor algebra generated by the trivial representation in degree $1$, where the product is the Day convolution.

When $d=1$, this is the twisted algebra $H_{0}\big(F_{\bullet}(\R)\big)$.
We prove that this twisted algebra is not Noetherian, i.e., a submodule of finitely generated module over $H_{0}\big(F_{\bullet}(\R)\big)$ need not be finitely generated.

\begin{prop}\label{not Noetherian}
As a twisted algebra, $H_{0}\big(F_{\bullet}(\R)\big)$ is not Noetherian.
\end{prop}

We prove Proposition \ref{not Noetherian} by constructing a non-finitely generated $H_{0}\big(F_{\bullet}(\R)\big)$-submodule of a finitely generated $H_{0}\big(F_{\bullet}(\R)\big)$-module.

\begin{proof}
As module over itself, $H_{0}\big(F_{\bullet}(\R)\big)$ is finitely generated in degree $0$.
We prove there is an $H_{0}\big(F_{\bullet}(\R)\big)$-submodule of $H_{0}\big(F_{\bullet}(\R)\big)$ with a generator in each positive even degree that cannot be finitely generated as an $H_{0}\big(F_{\bullet}(\R)\big)$-module.
 
Note that for all $m$
\[
\sum_{n=1}^{\frac{m}{2}}\frac{1}{\sqrt{(2n)!}}<1.
\]
Since the maximum dimension of an irreducible $S_{2n}$-representation is at most $\sqrt{(2n)!}$, it follows that there is an $H_{0}\big(F_{\bullet}(\R)\big)$-submodule $U$ of $H_{0}\big(F_{\bullet}(\R)\big)$ with one generator $[x_{2n}]$ in each of the positive even degrees, such that the $S_{2n}$-span of the degree $2n$ generator is a single irreducible $S_{2n}$-representation, and such that each generator is not in the $H_{0}\big(F_{\bullet}(\R)\big)$-submodule of $H_{0}\big(F_{\bullet}(\R)\big)$ generated by the lower degree generators.

Let $W$ be a finite FB-module such that there exists a surjection
\[
M^{H_{0}\big(F_{\bullet}(\R)\big)}(W)\twoheadrightarrow U\subset H_{0}\big(F_{\bullet}(\R)\big),
\] 
where $M^{H_{0}\big(F_{\bullet}(\R)\big)}(W)$ is the free $H_{0}\big(F_{\bullet}(\R)\big)$-module generated by $W$---this is equivalent to $M^{\text{FI}_{1,o}}(W)$---and let $m$ be largest value such that $W_{m}\neq 0$.
By definition, the map
\[
M^{H_{0}\big(F_{\bullet}(\R)\big)}(W)_{m}\to U_{m}\subset H_{0}\big(F_{n}(\R)\big),
\] 
is a surjection.
Let $l>m$ be the next even number bigger than $m$.
Then the maps
\[
M^{H_{0}\big(F_{\bullet}(\R)\big)}(W)_{l}\to U_{l}\subset H_{0}\big(F_{l}(\R)\big)
\]
and
\[
M^{H_{0}\big(F_{\bullet}(\R)\big)}(U_{m})_{l}\to U_{l}\subset H_{0}\big(F_{l}(\R)\big)
\]
have the same image in $U_{l}$, namely the classes that can be obtained by adding $l-m$ points to classes in $U_{m}$ and permuting the labels.
Since the classes in $U_{m}$ can be obtained by adding $m-2n$ points to the various $[x_{2n}]$, it follows that the image of $M^{H_{0}\big(F_{\bullet}(\R)\big)}(W)_{l}$ does not include the class $[x_{l}]$.
Therefore $W$ does generate $U$, a contradiction, so $U$ cannot be finitely generated.
\end{proof}

It follows from Proposition \ref{not Noetherian} that one cannot use a direct analogue of the technique used by An--Drummond-Cole--Knudsen to prove their Theorem 1.1 \cite[Theorem 1.1]{an2020edge} to prove that $H_{1}\big(F_{\bullet}(\Gamma_{k})\big)$ is a finitely generated FI$_{k,o}$-module.

For more on modules over a twisted algebra see \cite{sam2012introduction}.
\end{remark}

Having proved that $H_{i}\big(F_{\bullet}(\Gamma_{k})\big)$ is an FI$_{k,o}$-module, we turn to calculating its generation and presentation degrees.
In order to do so, we recall the definition of a cubical complex defined by L\"{u}tgehetmann that is homotopy equivalent to $F_{n}(\Gamma_{k})$.

\section{$F_{n}(\Gamma)$ and $K_{n}(\Gamma)$}\label{cube complex}

The focus of this paper is the first homology of the ordered configuration space of particles in the star graph $\Gamma_{k}$.
Instead of directly studying $F_{n}(\Gamma_{k})$, we study a combinatorial model $K_{n}(\Gamma)$ of ordered graph configuration spaces due to L\"{u}tgehetmann \cite{lu14}, which generalizes a model for unordered configuration spaces due to {\'S}wi{\k{a}}tkowski \cite{swikatkowski2001estimates}.
This combinatorial model has several nice properties that we will leverage through a spectral sequence argument.

Given a graph $\Gamma$, we describe the cube complex $K_{n}(\Gamma)$. 
The $0$-cells of $K_{n}(\Gamma)$ correspond to configurations of $n$ particles in $\Gamma$ such that every particle in the configuration is either on a vertex of $\Gamma$ or the interior of an edge of $\Gamma$ with the following caveat: if $e$ is an edge of $\Gamma$ connected to a leaf $v$, then a particle can only be on $e$ if one is already on $v$, i.e., we ensure the outermost particle on $e$ is moved to the leaf if there is not already a particle on the leaf.

The $1$-cells of $K_{n}(\Gamma)$ correspond to the movement of the closest particle on an edge of $\Gamma$ to a vertex to that vertex, if it is unoccupied.
Such a $1$-cell connects the two $0$-cells where this particle is at the vertex and where this particle is the particle on that edge closest to that vertex.
Additionally, if $v$ is a leaf and there is no particle on the edge connecting $v$ to the rest of the graph and no particle at the neighboring vertex, there is $1$ cell in $K_{n}(\Gamma)$ connecting the corresponding $0$-cells.

In general, the $m$-cells of $K_{n}(\Gamma)$ correspond to the movement of $m$ particles in $\Gamma$, closest on their edge to the unoccupied essential vertex to which they are moving.
Moreover, no two particles are allowed to go to the same vertex.
See Figure \ref{K2Gamma3} for an example of a $K_{n}(\Gamma)$.

We recall the following theorem of L\"{u}tgehetmann, which will allow us to study $K_{n}(\Gamma)$ in place of $F_{n}(\Gamma)$.

\begin{figure}[h]
\centering
\captionsetup{width=.8\linewidth}
\includegraphics[width = 10cm]{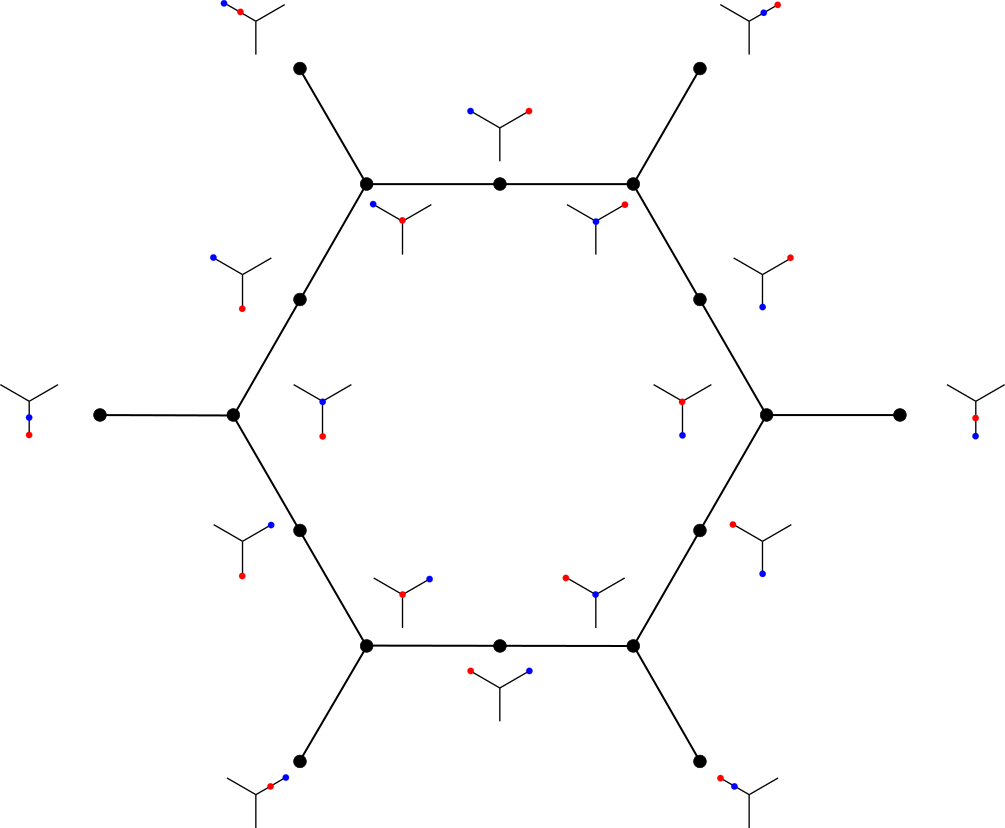}
\caption{The cubical complex $K_{2}(\Gamma_{3})$, which is an equivariant deformation retract of $F_{2}(\Gamma_{3})$.
}
\label{K2Gamma3}
\end{figure}

\begin{thm}\label{cube complex for configuration space}
(L\"{u}tgehetmann \cite[Theorem 2.3]{lu14}) Let $\Gamma$ be a locally finite graph, then the finite-dimensional cube complex $K_{n}(\Gamma)$ is an $S_{n}$-equivariant deformation retract of $F_{n}(\Gamma)$.
If $\Gamma$ is finite, then $K_{n}(\Gamma)$ consists of finitely many cells.
\end{thm}

In particular, if $\Gamma=\Gamma_{k}$ it follows from the definition that $K_{n}(\Gamma_{k})$ is a $1$-dimensional complex, i.e., a graph; we parametrize it so that its edges have length $1$.
We state for convenience the following well-known result that will limit our work to two cases.

\begin{prop}\label{only homology in degree 0,1}
If $i\neq 0, 1$, then $H_{i}\big(F_{n}(\Gamma_{k})\big)=0$.
\end{prop}

Next, we state a result due to Ghrist \cite{Ghr01} describing the structure of $\pi_{1}\big(F_{n}(\Gamma_{k})\big)$.
Later, we will use an immediate corollary of this proposition to find a bound on the generation degree of $H_{1}\big(F_{\bullet}(\Gamma_{k})\big)$ as an FI$_{k,o}$-module.

\begin{prop}\label{pi1 of conf}
(Ghrist \cite[Proposition 4.1]{Ghr01})
The braid group $\pi_{1}\big(F_{n}(\Gamma_{k})\big)$ is isomorphic to a free group on $Q$ generators, where
\[
Q=1+(nk-2n-k+1)\frac{(n+k-2)!}{(k-1)!}.
\]
\end{prop}

This proposition immediately implies that
\[
\text{rk}\Big(H_{1}\big(F_{n}(\Gamma_{k})\big)\Big)=1+(nk-2n-k+1)\frac{(n+k-2)!}{(k-1)!}.
\]

In the next section we recall the Mayer--Vietoris spectral sequence which we will use to analyze $K_{n+1}(\Gamma_{k})$ to get our representation stability results.

\section{The Mayer--Vietoris Spectral Sequence}\label{mayer vietoris}

We recall the Mayer--Vietoris spectral sequence, which we will use to prove our representation stability results.
After choosing a sufficiently nice cover of our configuration spaces, the vanishing of the $E^{2}_{1,0}$-entry of this spectral sequence will correspond to $H_{1}\big(F_{n+1}(\Gamma_{k})\big)$ being generated by the classes of $H_{1}\big(F_{n}(\Gamma_{k})\big)$ tensored with the fundamental class of $H_{0}\big(F_{1}(\R)\big)$ in $k$ different ways, i.e., $H_{1}\big(F_{n+1}(\Gamma_{k})\big)$ is generated by $H_{1}\big(F_{n}(\Gamma_{k})\big)$ and the FI$_{k,o}$-module structure of $H_{1}\big(F_{\bullet}(\Gamma_{k})\big)$.
Moreover, the vanishing of the $E^{2}_{2,0}$-entry of this spectral sequence will correspond to $H_{1}\big(F_{\bullet}(\Gamma_{k})\big)$ being a finitely presented FI$_{k,o}$-module.

Let $I$ be a countable index set, and let $\mathcal{U}=\{U_{i}\}_{i\in I}$ be an open cover of a space $X$.
The Mayer--Vietoris spectral sequence has $E^{1}$-page
\[
E^{1}_{p,q}=\bigoplus_{\{i_{0}, \dots, i_{p}\}}H_{q}(U_{i_{0}}\cap\cdots\cap U_{i_{p}})\implies H_{p+q}(X)
\]
with boundary map $d_{1}:E^{1}_{p,q}\to E^{1}_{p-1, q}$ given by the alternating sum of the face maps induced by
\[
U_{i_{0}}\cap\cdots\cap U_{i_{p}}\hookrightarrow U_{i_{0}}\cap\cdots\cap \widehat{U_{i_{j}}}\cap\cdots\cap U_{i_{p}}.
\]

\begin{figure}[h]
\centering
\begin{tikzpicture} \footnotesize
  \matrix (m) [matrix of math nodes, nodes in empty cells, nodes={minimum width=3ex, minimum height=5ex, outer sep=2ex}, column sep=3ex, row sep=3ex]{
 3    &  \bigoplus_{\{i_{0}\}} H_{3}(U_{i_{0}})  &\bigoplus_{\{i_{0}, i_{1}\}} H_{3}(U_{i_{0}}\cap U_{i_{1}}) &  \bigoplus_{\{i_{0}, i_{1}, i_{2}\}} H_{3}(U_{i_{0}}\cap U_{i_{1}}\cap U_{i_{2}})  & \bigoplus_{\{i_{0}, \dots, i_{3}\}} H_{3}(U_{i_{0}}\cap \cdots\cap U_{i_{3}})& \\  
 2    &  \bigoplus_{\{i_{0}\}} H_{2}(U_{i_{0}})  &\bigoplus_{\{i_{0}, i_{1}\}} H_{2}(U_{i_{0}}\cap U_{i_{1}}) &  \bigoplus_{\{i_{0}, i_{1}, i_{2}\}} H_{2}(U_{i_{0}}\cap U_{i_{1}}\cap U_{i_{2}})  & \bigoplus_{\{i_{0}, \dots, i_{3}\}} H_{2}(U_{i_{0}}\cap \cdots\cap U_{i_{3}})&  \\          
 1    &  \bigoplus_{\{i_{0}\}} H_{1}(U_{i_{0}})  &\bigoplus_{\{i_{0}, i_{1}\}} H_{1}(U_{i_{0}}\cap U_{i_{1}}) &  \bigoplus_{\{i_{0}, i_{1}, i_{2}\}} H_{1}(U_{i_{0}}\cap U_{i_{1}}\cap U_{i_{2}})  & \bigoplus_{\{i_{0}, \dots, i_{3}\}} H_{1}(U_{i_{0}}\cap \cdots\cap U_{i_{3}})& \\             
  0    &  \bigoplus_{\{i_{0}\}} H_{0}(U_{i_{0}})  &\bigoplus_{\{i_{0}, i_{1}\}} H_{0}(U_{i_{0}}\cap U_{i_{1}}) &  \bigoplus_{\{i_{0}, i_{1}, i_{2}\}} H_{0}(U_{i_{0}}\cap U_{i_{1}}\cap U_{i_{2}})  & \bigoplus_{\{i_{0}, \dots, i_{3}\}} H_{0}(U_{i_{0}}\cap \cdots\cap U_{i_{3}})& \\       
 \quad\strut &  0  &  1  & 2  &3&\\}; 

\draw[thick] (m-1-1.east) -- (m-5-1.east) ;
\draw[thick] (m-5-1.north) -- (m-5-5.north east) ;
\end{tikzpicture}
\caption{The $E^{1}$-page of the Mayer--Vietoris spectral sequence.}
\label{MayerVietoris}
\end{figure}

The $E^{2}_{1,0}$- and $E^{2}_{2,0}$-entries of this spectral sequence will be of great interest to us.
As such, we relate them to a simplicial complex arising from the open cover of configuration space.

Given a set of indices $I$ and $\mathcal{U}=\{U_{i}\}_{i\in I}$ a family of sets, the \emph{pseudo-nerve complex} $N(\mathcal{U})$ of $\mathcal{U}$ is the abstract $\Delta$-complex whose $p$-simplices are the path-components of the intersection of $p+1$ distinct elements of $\mathcal{U}$, i.e., $U_{i_{0}}\cap \cdots\cap U_{i_{p}}$.
By ordering $I$, we can define a boundary map on this simplicial complex by setting
\[
d(U_{i_{0}}\cap \cdots\cap U_{i_{p}})=\sum_{j=0}^{p}(-1)^{k-1}U_{i_{0}}\cap\cdots\cap \widehat{U_{i_{j}}}\cap\cdots\cap U_{i_{p}},
\]
for each component of $U_{i_{0}}\cap \cdots\cap U_{i_{p}}$.
If all of the intersections are empty or path-connected, we call the resulting simplicial complex the \emph{nerve complex} of $\mathcal{U}$.

The following proposition immediately follows from the definitions, and relates the Mayer--Vietoris spectral sequence to the homology of the (pseudo)-nerve complex given that the cover is sufficiently nice.

\begin{prop}\label{E2 of mayer--vietoris}
Let $J$ be a countable index set and $\mathcal{U}=\{U_{j}\}_{j\in J}$ an open cover of $X$ such that each intersection $U_{i_{0}}\cap\cdots\cap U_{i_{p}}$ is either empty or path-connected.
Then the $E^{2}_{p,0}$-entry of the Mayer--Vietoris spectral sequence is the $p^{\text{th}}$-homology of $N(\mathcal{U})$.
\end{prop}

For more on the Mayer--Vietoris spectral sequence see, for example, \cite[Chapter 7.4]{brown2012cohomology}.

In the next section we will find a good cover of $K_{n+1}(\Gamma_{k})$ by sets that are homotopy equivalent to $K_{n}(\Gamma_{k})$ and calculate the first homology of the resulting nerve complex.
This will allow us to use Proposition \ref{E2 of mayer--vietoris} to determine when the first homology of $F_{n+1}(\Gamma_{k})$ is governed by the $E^{1}_{0,1}$-term, a condition equivalent to $H_{1}\big(F_{n+1}(\Gamma_{k})\big)$ being generated by $H_{1}\big(F_{n}(\Gamma_{k})\big)$ and the FI$_{k,o}$-module structure of $H_{1}\big(F_{\bullet}(\Gamma_{k})\big)$.

\section{Representation Stability for $H_{i}\big(F_{\bullet}(\Gamma_{k})\big)$}\label{rep stab section}

We prove that the homology of $F_{\bullet}(\Gamma_{k})$ satisfies a notion of representation stability, namely that $H_{i}\big(F_{\bullet}(\Gamma_{k})\big)$ is a finitely generated FI$_{k,o}$-module, and calculate its generation degree as such.
When $i\neq 1$, this problem is easy.
Proposition \ref{only homology in degree 0,1} shows that if $i>1$, then $H_{i}\big(F_{\bullet}(\Gamma_{k})\big)$ is the trivial FI$_{k,o}$-module.
Since $F_{n}(\Gamma_{k})$ is connected and adding a new point to configuration space yields a nontrivial homology class, it follows that $H_{0}\big(F_{\bullet}(\Gamma_{k})\big)$ is finitely generated in degree $0$ as an FI$_{k,o}$-module.
We recall the statement of our representation stability results for all $i$ and spend the rest of the section proving it in the case $i=1$.

\begin{T1}
  \thmtext
\end{T1} 

In order to prove Theorem \ref{rep stability for star} for $i=1$, we take advantage of the Mayer--Vietoris spectral sequence to compute the first homology of $F_{n+1}(\Gamma_{k})$ by covering $F_{n+1}(\Gamma_{k})$ with $k(n+1)$ copies of $F_{n}(\Gamma_{k})$, each of which corresponds to adding a point at one of the $k$ leaves of $\Gamma_{k}$.
We prove that the vanishing of the $E^{2}_{1,0}$-entry of this spectral sequence for $n$ sufficiently large is equivalent to $H_{1}\big(F_{\bullet}(\Gamma_{k})\big)$ being a finitely generated FI$_{k,o}$-module.
Throughout this section we will use $F_{n}(\Gamma_{k})$ interchangeably with its $1$-dimensional cube complex $K_{n}(\Gamma_{k})$.

\begin{prop}\label{retract to get good cover}
For $i\in [n+1]$ and $j\in [k]$, let $U_{i,j}$ be the subspace of $F_{n+1}(\Gamma_{k})$ consisting of configurations where particle $i$ is the outermost particle on edge $j$.
Fix $0<\epsilon<\frac{1}{2}$, and extend $U_{i,j}$ to $U_{i,j}'$ by taking the open $\epsilon$-ball around $U_{i,j}$.
Then $U_{i,j}$ is a deformation retract of $U_{i,j}'$, and $U_{i,j}$ is homotopy equivalent to $F_{n}(\Gamma_{k})$.
\end{prop}

\begin{proof}
That $U_{i,j}$ is a deformation retract of $U_{i,j}'$ is immediate from the fact that $F_{n+1}(\Gamma_{k})$ can be viewed as a graph with edges of length $1$.

To see that $U_{i,j}$ is homotopy equivalent to $F_{n}(\Gamma_{k})$, note that forgetting the particle labeled $i$ in $U_{i,j}$, yields a configuration space on $n$ points in $\Gamma_{k}$.
Since particle $i$ is the outermost particle on its edge we may assume that it is at the leaf.
Thus, forgetting it yields a space homeomorphic to $\Gamma_{k}$, implying that $U_{i,j}$ is homotopy equivalent to $F_{n}(\Gamma_{k})$.
\end{proof}

Note that the $U_{i,j}$ correspond to the image of $F_{n}(\Gamma_{k})$ in $F_{n+1}(\Gamma_{k})$ under the insertion map $\iota_{n, j}$ and an action of $S_{n+1}$.
See Figure \ref{U11} for an example of a $U_{i,j}$ in $F_{2}(\Gamma_{3})$.

\begin{figure}[h]
\centering
\captionsetup{width=.8\linewidth}
\includegraphics[width = 12cm]{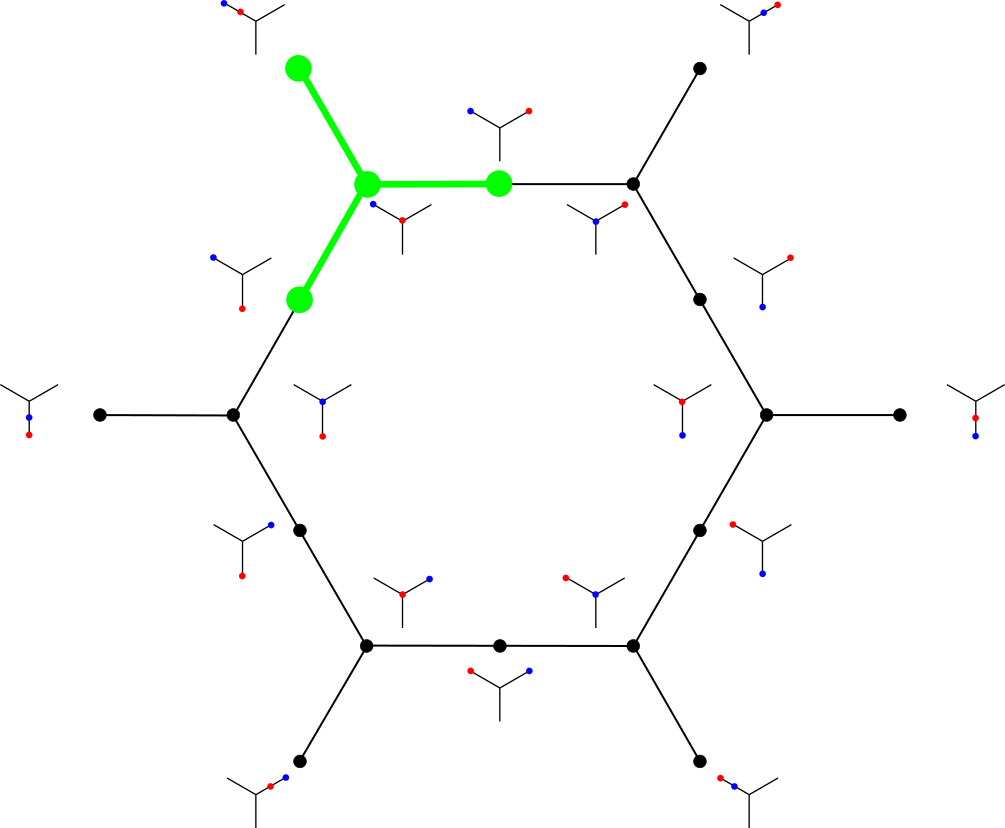}
\caption{The green part of $F_{2}(\Gamma_{3})$ is the region $U_{1, 1}$, where blue particle can be thought of as having label $1$.
Note that this region is homotopy equivalent to $F_{1}(\Gamma_{3})$.
}
\label{U11}
\end{figure}

\begin{prop}\label{good cover}
Let $U_{i,j}'$ be the open sets described in Proposition \ref{retract to get good cover}.
Then, for $n\ge 1$, the collection $\mathcal{U}_{n+1, k}:=\{U'_{i,j}\}_{i\in [n+1], j\in [k]}$ is an open cover of $F_{n+1}(\Gamma_{k})$.
\end{prop}

\begin{proof}
Since $n\ge 1$, every point in $F_{n+1}(\Gamma_{k})$ corresponds to a configuration where a particle is on some edge of $\Gamma_{k}$ as only one particle can be at the central vertex.
Given such an edge $j$, there is an outermost particle $i$ on $j$.
This configuration is in $U'_{i,j}$, so $\mathcal{U}$ is an open cover of $F_{n+1}(\Gamma_{k})$.
\end{proof}

From now on we will write $E^{r}_{p,q}[\Gamma_{k}](n+1)$ for the Mayer--Vietoris spectral sequence arising from the open cover $\mathcal{U}_{n+1, k}$ of $F_{n+1}(\Gamma_{k})$.

Every non-trivial intersection $U'_{i_{0}, j_{0}}\cap\cdots \cap U'_{j_{l}, i_{l}}$ is connected, as it is homotopy equivalent to $F_{n-l}(\Gamma_{k})$: It is the space of configurations where $l+1$ distinct particles $i_{0}, \dots, i_{l}$ are the outermost particles on $l+1$ distinct edges $j_{0}, \dots, j_{l}$, and the remaining $n-l$ particles may be anywhere else in $\Gamma_{k}$ as long as they are not past any of the particles $i_{l}$, i.e., they can move about a graph homeomorphic to $\Gamma_{k}$.
It follows that $N(\mathcal{U}_{n+1, k})^{(1)}$ is a \emph{simple} graph, i.e., there is at most one edge between any two vertices and there are no loops.
With that in mind, we note the following.

\begin{prop}\label{3 cycles are contractible}
For all $k\ge 3$ and all $n\ge 2$, every $1$-cycle of length $3$ in $N(\mathcal{U}_{n+1, k})$ is null-homotopic.
\end{prop}

\begin{proof}
For $(i,j)\neq(i', j')$, intersection of $U_{i,j}'$ and $U'_{i',j'}$ is non-trivial if and only if $i\neq i'$ and $j\neq j'$, as if $j=j'$, then the outermost particle on edge $j$ of $\Gamma_{k}$ would have to be labeled both $i$ and $i'$, and if $i=i'$, then the particle labeled $i$ would have to be the outermost particle on edges $j$ and $j'$.
Therefore, a length $3$ cycle in $N(\mathcal{U}_{n+1, k})$ must be of the form
\[
U_{i,j}'\cap U_{i',j'}'-U_{i',j'}'\cap U_{i'',j''}'+U_{i'',j''}'\cap U_{i,j}',
\]
where $i$, $i'$, and $i''$ are distinct particles and $j$, $j'$, and $j''$ are distinct edges.
This is the case in any $1$-cycle of length $3$, so since $k\ge 3$ and $n+1\ge 3$, we have that $U_{i,j}'\cap U_{i',j'}'\cap U_{i'',j''}'\neq \emptyset$, as the space of configurations where particle $i$ is the outermost particle on edge $j$, particle $i'$ is the outermost particle on edge $j'$, and particle $i''$ is the outermost particle on edge $j''$ is non-empty, see Figure \ref{length3cycleinnerve}.
Moreover,
\[
U_{i,j}'\cap U_{i',j'}'-U_{i',j'}'\cap U_{i'',j''}'+U_{i'',j''}'\cap U_{i,j}'=d(U_{i,j}'\cap U_{i',j'}'\cap U_{i'',j''}'),
\]
in $N(\mathcal{U}_{n+1, k})$, so this $1$-cycle of length $3$ is null-homotopic in $N(\mathcal{U}_{n+1, k})$.
\end{proof}

\begin{figure}[h]
\centering
\captionsetup{width=.8\linewidth}
\includegraphics[width = 8cm]{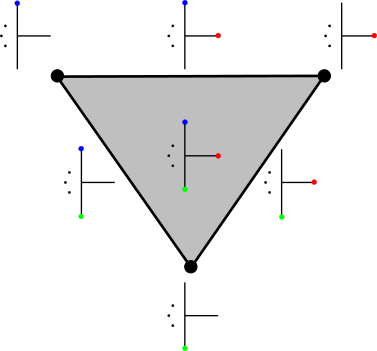}
\caption{The $2$-cell in $N(\mathcal{U}_{n+1, k})$ that makes $1$-cycles of length $3$ null-homotopic.
}
\label{length3cycleinnerve}
\end{figure}

Next, we note that there are only five types of length $3$ paths in $N(\mathcal{U}_{n+1, k})^{(1)}$ that are not immediately homotopy equivalent to a length $2$ path in $N(\mathcal{U}_{n+1, k})$. 
We will show that each of these five paths are in fact homotopy equivalent to a path of length $2$, allowing us to apply an inductive argument to show that for large enough $n$ all $1$-cycles are null-homotopic in $N(\mathcal{U}_{n+1, k})$.

\begin{prop}\label{5 types of paths}
For all $k\ge 3$, there are five types of paths of length $3$ in $N(\mathcal{U}_{n+1, k})$ such that no three consecutive distinct vertices correspond to open sets $U'_{i,j}$, $U'_{i',j'}$, and $U'_{i'',j''}$ with non-trivial intersection $U_{i,j}'\cap U_{i',j'}'\cap U_{i'',j''}'$.
\end{prop}

The five types of paths of length 3 in $N(\mathcal{U}_{n+1, k})^{(1)}$ that are not immediately homotopy equivalent to a path of length $2$ in $N(\mathcal{U}_{n+1, k})$ are of the following forms:
\begin{enumerate}
\item The third vertex $U'_{i_{3}, j_{3}}$ is such that $i_{3}=i_{1}$ and the fourth vertex $U'_{i_{4}, j_{4}}$ is such that $i_{4}=i_{2}$, Moreover, $j_{1},j_{2},$ and $j_{3}$ are distinct and so are $j_{2}$, $j_{3}$, $j_{4}$.
\item The third vertex $U'_{i_{3}, j_{3}}$ is such that $i_{3}=i_{1}$, the fourth vertex $U'_{i_{4}, j_{4}}$ is such that $i_{2}$, $i_{3}$, and $i_{4}$ are all distinct.
Additionally, $j_{1}$, $j_{2}$, and $j_{3}$ are distinct and $j_{2}=j_{4}$.
\item The third vertex $U'_{i_{3}, j_{3}}$ is such that $i_{3}\neq i_{1}$ and $j_{3}=j_{1}$, and the fourth vertex $U'_{i_{4}, j_{4}}$ is such that $i_{4}=i_{2}$ and $j_{2}$, $j_{3}$, and $j_{4}$ are distinct.
\item The third vertex $U'_{i_{3}, j_{3}}$ is such that $i_{3}\neq i_{1}$ and $j_{3}=j_{1}$, and the fourth vertex $U'_{i_{4}, j_{4}}$ is such that $i_{4}=i_{1}$ and $j_{4}=j_{2}$.
\item The third vertex $U'_{i_{3}, j_{3}}$ is such that $i_{3}\neq i_{1}$ and $j_{3}=j_{1}$, and the fourth vertex $U'_{i_{4}, j_{4}}$ is such that $i_{4}$ is distinct from  $i_{1}$, $i_{2}$, and $i_{3}$, which are all distinct, and $j_{4}=j_{2}$.
\end{enumerate}
See Figure \ref{typesoflength4pathinnerve} for a visual representation of these five paths.

\begin{figure}[h]
\centering
\captionsetup{width=.8\linewidth}
\includegraphics[width = 12cm]{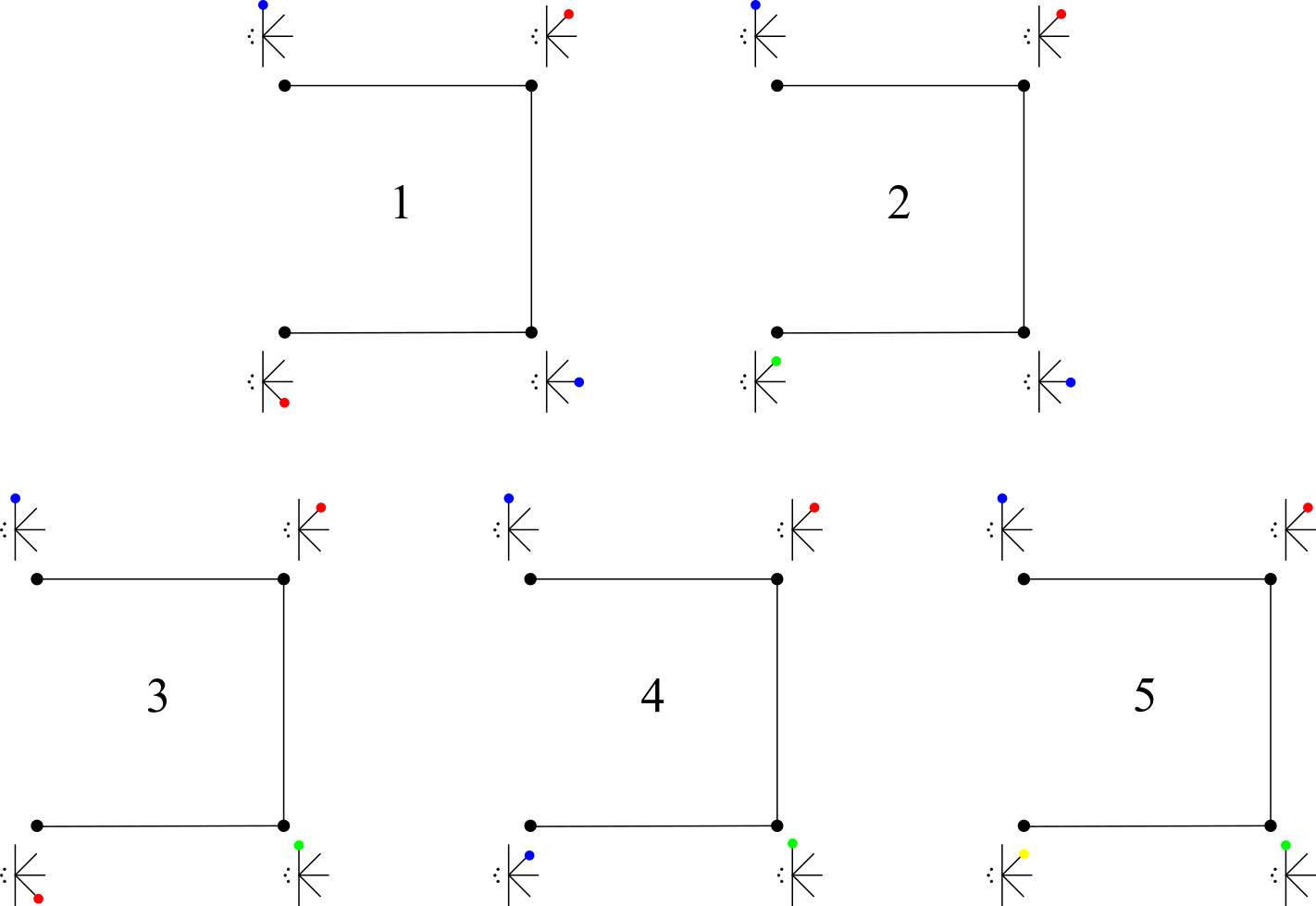}
\caption{The five possible paths of length $3$ in $N(\mathcal{U}_{n+1, k})^{(1)}$ that are not immediately homotopy equivalent to a path of length $2$ in $N(\mathcal{U}_{n+1, k})$.
There are several more types of paths of length $3$ in $N(\mathcal{U}_{n+1, k})^{(1)}$, though for each such path either the first three vertices or the last three vertices lie in the same $2$-simplex in $N(\mathcal{U}_{n+1, k})$, yielding a shorter homotopy equivalent path.
}
\label{typesoflength4pathinnerve}
\end{figure}

Next, we determine for what values of $k$ and $n$ can we shorten paths of length $3$ in $N(\mathcal{U}_{n+1, k})$ to paths of length $2$.
This will show us when $1$-cycles in $N(\mathcal{U}_{n+1, k})$ are null-homotopic.

\begin{prop}\label{shorten paths of length 3}
For $k=3$ and $n\ge 4$, $k=4$ and $n\ge 3$, and $k\ge 5$ and $n\ge 2$, every path of length $3$ in $N(\mathcal{U}_{n+1, k})$ is homotopy equivalent to a path of length $2$ in $N(\mathcal{U}_{n+1, k})$.
\end{prop}

\begin{proof}
We only need to consider paths $\gamma$ of length $3$ whose vertices are distinct such that for any three consecutive vertices corresponding to sets $U'_{i,j}$, $U'_{i',j'}$, and $U'_{i'', j''}$ we have that $U'_{i,j}\cap U'_{i',j'}\cap U'_{i'', j''}=\emptyset$, as if $U'_{i,j}\cap U'_{i',j'}\cap U'_{i'', j''}\neq\emptyset$, then Proposition \ref{3 cycles are contractible} proves we can shorten our path.

It follows that we only need to consider the $5$ types of paths following Proposition \ref{5 types of paths}.
We prove that for each such $\gamma$ there is a path $\gamma'$ of length $2$ in $N(\mathcal{U}_{n+1, k})$ such that $\gamma-\gamma'$ is the boundary of a $2$-chain in $N(\mathcal{U}_{n+1, k})$.
We do so for $k=3$ and $n\ge 4$ in Figure \ref{homologouspathnerveGamma3}, for $k=4$ and $n\ge 3$ in Figure \ref{homologouspathnerveGamma4}, and for $k\ge 5$ and $n\ge 2$ in Figure \ref{homologouspathnerveGamma5}.

\begin{figure}[h]
\centering
\captionsetup{width=.8\linewidth}
\includegraphics[width = 12cm]{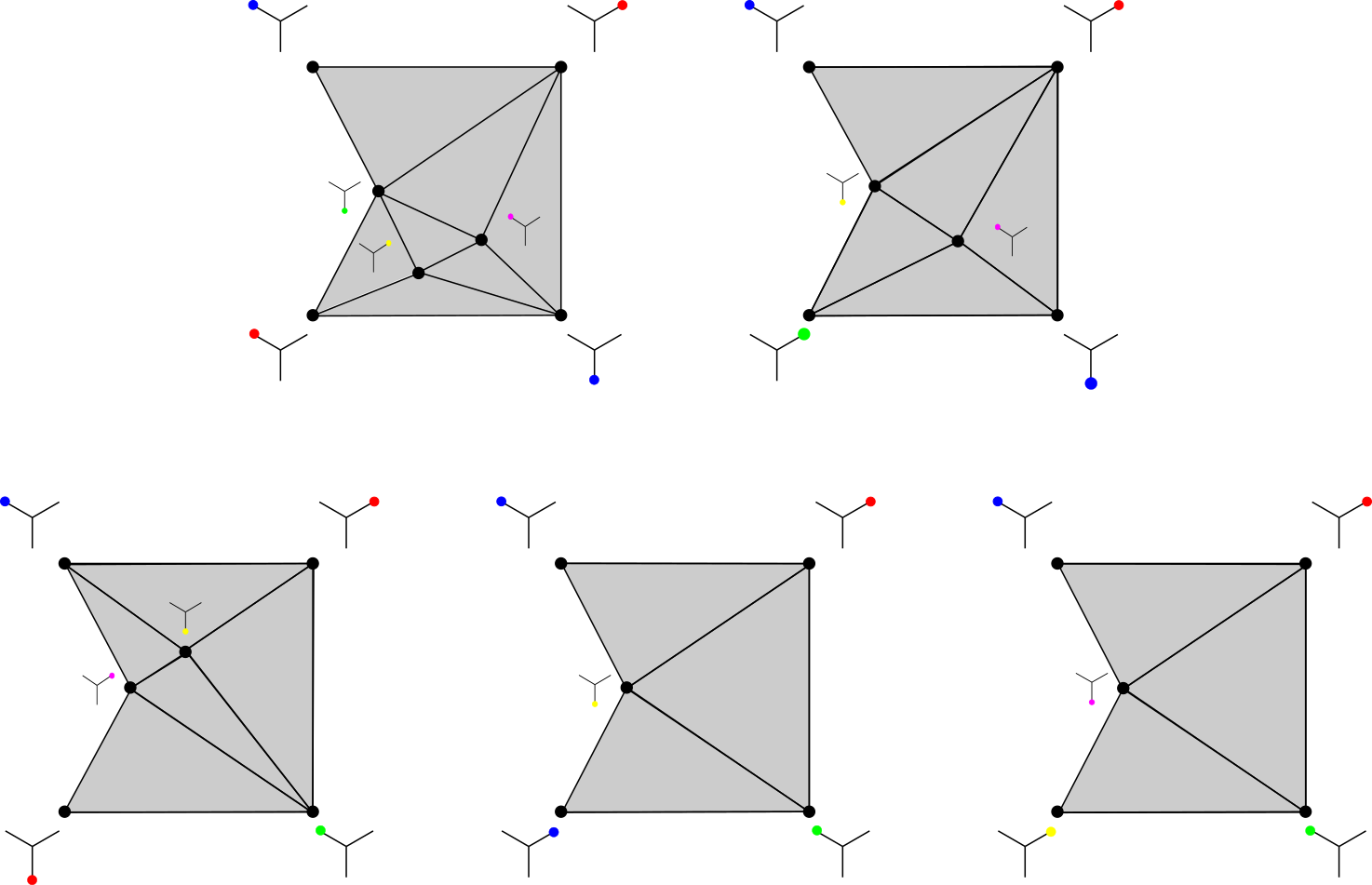}
\caption{The collections of $2$-cells in $N(\mathcal{U}_{n+1, k})$ that allow us to shorten the $5$ paths of Proposition \ref{5 types of paths} when $k=3$ and $n\ge4$.
}
\label{homologouspathnerveGamma3}
\end{figure}

\begin{figure}[h]
\centering
\captionsetup{width=.8\linewidth}
\includegraphics[width = 12cm]{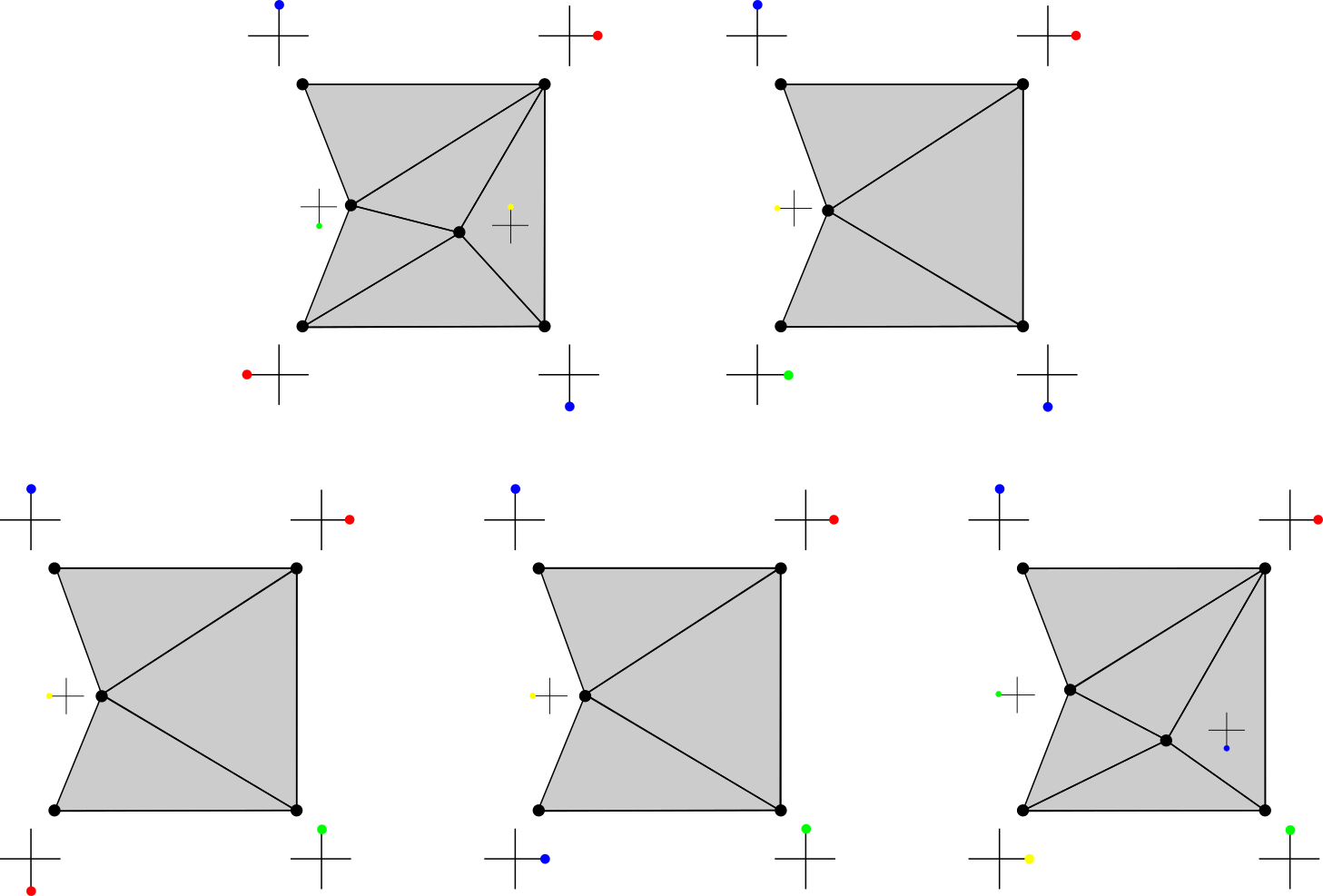}
\caption{The collections of $2$-cells in $N(\mathcal{U}_{n+1,4})$ that allow us to shorten the $5$ paths of Proposition \ref{5 types of paths} when $k=4$ and $n\ge3$.
}
\label{homologouspathnerveGamma4}
\end{figure}

\begin{figure}[h]
\centering
\captionsetup{width=.8\linewidth}
\includegraphics[width = 12cm]{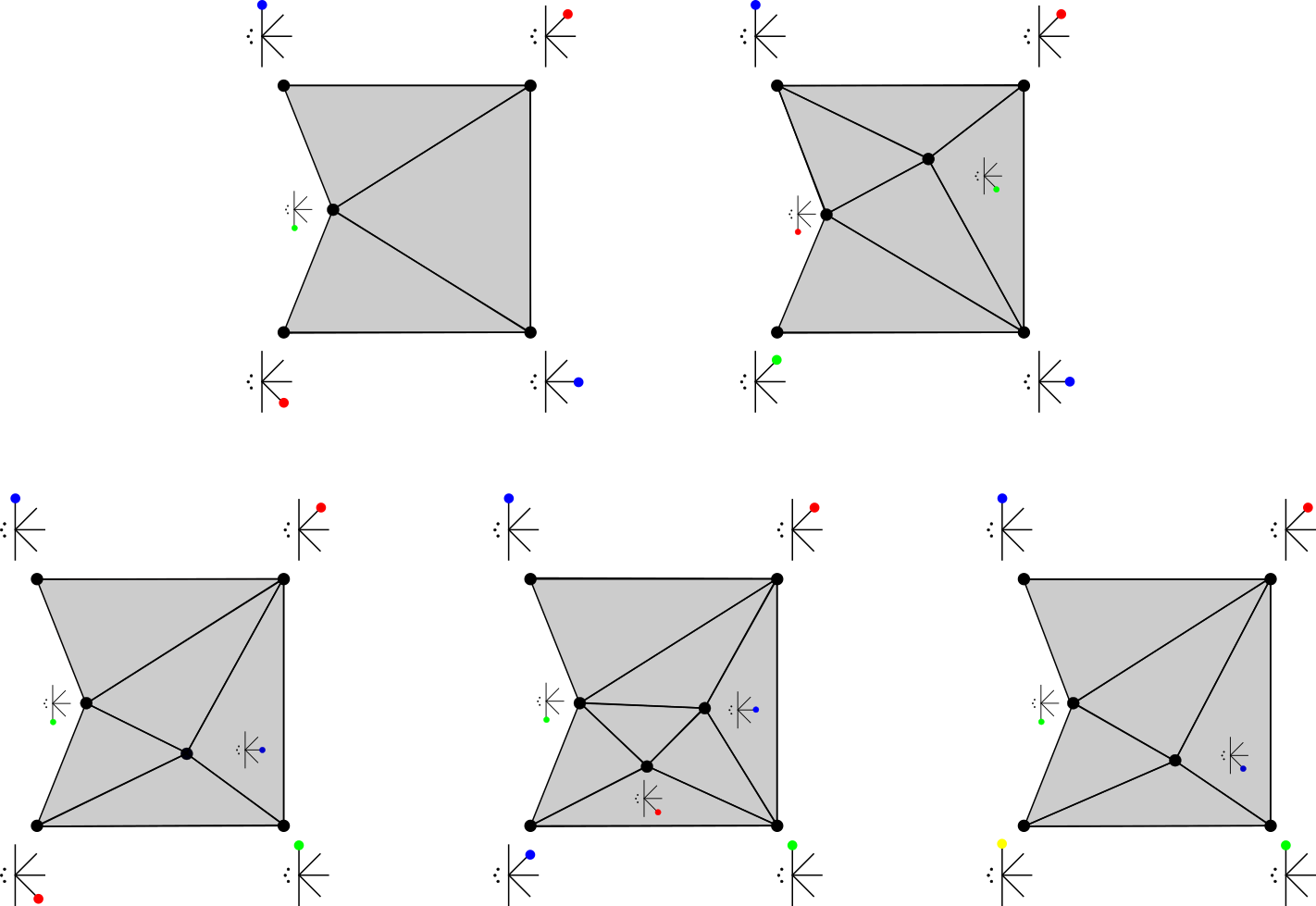}
\caption{The collections of $2$-cells in $N(\mathcal{U}_{n+1, k})$ that allow us to shorten the $5$ paths of Proposition \ref{5 types of paths} when $k\ge 5$ and $n\ge2$.
}
\label{homologouspathnerveGamma5}
\end{figure}

\end{proof}

We use Proposition \ref{shorten paths of length 3} to provide conditions for $N(\mathcal{U}_{n+1, k})$ to have trivial first homology.
This will allow us to use Proposition \ref{E2 of mayer--vietoris} to show that the $E^{2}_{1,0}[\Gamma_{k}](n+1)$-entry of the Mayer--Vietoris spectral sequence is $0$, proving our finite generation results.

\begin{lem}\label{contract 1-cycles}
If $k=3$ and $n\ge 4$, or $k=4$ and $n\ge 3$, or $k\ge 5$ and $n\ge 2$, then $H_{1}\big(N(\mathcal{U}_{n+1, k})\big)=0$.
\end{lem}

\begin{proof}
We proceed by induction on the length of a $1$-cycle in $N(\mathcal{U}_{n+1, k})$. 
Since $N(\mathcal{U}_{n+1, k})^{(1)}$ is a simple graph there are no $1$-cycles of length $1$ in $N(\mathcal{U}_{n+1, k})$ and all $1$-cycles of length $2$ are trivially null-homotopic.
By Proposition \ref{3 cycles are contractible} all $1$-cycles of length $3$ in $N(\mathcal{U}_{n+1, k})$ are null-homotopic.

Assume all $1$-cycles of length at most $m$ are null-homotopic in $N(\mathcal{U}_{n+1, k})$.
Let $\kappa$ be a $1$-cycle of length $m+1$ in $N(\mathcal{U}_{n+1, k})$.
We may assume that $\kappa$ does not self intersect as if it does, we can decompose $\kappa$ into two strictly shorter $1$-cycles, which are null-homotopic by assumption. 
Let $\gamma$ be a subpath of $\kappa$ of length $3$.
By Proposition \ref{shorten paths of length 3}, $\gamma$ is homotopy equivalent to a path $\gamma'$ of length $2$, implying that $\kappa$ is homotopy equivalent to a $1$-cycle $\kappa'$ of length $m$. 
Since $\kappa'$ is null-homotopic by the induction hypothesis, it follows that $\kappa$ is null-homotopic and $H_{1}\big(N(\mathcal{U}_{n+1, k})\big)=0$.
\end{proof}

Since the $E^{2}_{1,0}[\Gamma_{k}](n+1)$-entry of the Mayer--Vietoris spectral sequence is the first homology of $N(\mathcal{U}_{n+1, k})$ we get the following corollary.

\begin{cor}\label{E210 is 0}
If $k=3$ and $n\ge 4$, or $k=4$ and $n\ge 3$, or $k\ge 5$ and $n\ge 2$, then $E^{2}_{1,0}[\Gamma_{k}](n+1)$-entry of the Mayer--Vietoris spectral sequence is $0$.
\end{cor}

Finally, we use this to complete the proof of Theorem \ref{rep stability for star}.

\begin{proof}
First we prove that $H_{1}\big(F_{\bullet}(\Gamma_{k})\big)$ can be generated as an FI$_{k, o}$-module in the claimed degrees.
Then we confirm that these degrees are optimal.

It suffices to determine when all the classes in $H_{1}\big(F_{n+1}(\Gamma_{k})\big)$ can be written as sums of classes in $H_{1}\big(F_{n+1}(\Gamma_{k})\big)$ that arise from the $k$ ways of tensoring classes in $H_{1}\big(F_{n}(\Gamma_{k})\big)$ with the fundamental class of $H_{0}\big(F_{1}(\R)\big)$.
By the choice of our cover $\mathcal{U}_{n+1, k}$ of $F_{n+1}(\Gamma_{k})$, this is the same as showing that $H_{1}\big(F_{n+1}(\Gamma_{k})\big)$ is generated by $E^{1}_{0,1}[\Gamma_{k}](n+1)$.
Since the Mayer--Vietoris spectral sequence is trivial for $p, q<0$, this is equivalent to determining when $E^{2}_{1,0}[\Gamma_{k}](n+1)=0$, as this is the only other entry of the spectral sequence where new classes in first homology arise.

Let $E^{2}_{1, 0}[\Gamma_{k}]$ denote the FB-module, whose degree $(n+1)^{\text{th}}$-term is $E^{2}_{1,0}[\Gamma_{k}](n+1)$.
By Corollary \ref{E210 is 0}, we have that $E^{2}_{1,0}[\Gamma_{k}](n+1)=0$ for $n\ge 4$ and $k=3$, for $n\ge 3$ and $k=4$, and for $n\ge 2$ and $k\ge 5$.
It follows that as an FI$_{k,o}$-module, $H_{1}\big(F_{\bullet}(\Gamma_{k})\big)$ is finitely generated by $E^{2}_{1, 0}[\Gamma_{k}]$ in degree at most $4$ for $k=3$, degree at most $3$ for $k=4$, and degree at most $2$ for $k\ge 5$, i.e., there are surjections
\[
\bigoplus_{m\le N_{k}}M^{\text{FI}_{k, o}}\big(E^{2}_{1,0}[\Gamma_{k}](m)\big)\twoheadrightarrow H_{1}\big(F_{\bullet}(\Gamma_{k})\big),
\]
where if $k\ge 5$, we have $N_{k}=2$, if $k=4$, we have $N_{k}=3$, and if $k=3$, we have $N_{k}=4$, as in these cases we have $\bigoplus_{m\le N_{k}}M^{\text{FI}_{k, o}}\big(E^{2}_{1,0}[\Gamma_{k}](m)\big)=M^{\text{FI}_{k, o}}\big(E^{2}_{1,0}[\Gamma_{k}]\big)$.

Next, we show that $E^{2}_{1,0}[\Gamma_{k}](N_{k})\neq 0$.
To do this, we use the formula for the rank of $H_{1}\big(F_{n}(\Gamma_{k})\big)$ following from Proposition \ref{pi1 of conf}.
If $k\ge 5$, note that $E^{1}_{0,1}[\Gamma_{k}](2)=H_{1}\big(F_{1}(\Gamma_{k})\big)^{\oplus 2k}\cong0$, whereas $H_{1}\big(F_{2}(\Gamma_{k})\big)$ is non-trivial, so $E^{2}_{1,0}[\Gamma_{k}](2)\neq 0$.
If $k=4$, one may check that $E^{1}_{0,1}[\Gamma_{4}](3)=H_{1}\big(F_{2}(\Gamma_{4})\big)^{\oplus 12}\cong\Z^{60}$, but $H_{1}\big(F_{3}(\Gamma_{4})\big)\cong\Z^{61}$, so $E^{2}_{1,0}[\Gamma_{4}](3)\not\cong 0$.
Finally, if $k=3$, we have that $N(\mathcal{U}_{4,3})$ is a $2$-dimensional simplicial complex with $12$ $0$-cells, $36$ $1$-cells, and $24$ $2$-cells.
It follows from Betti number considerations and that the fact that every $1$-cell of $N(\mathcal{U}_{4,3})$ is contained in precisely two $2$-cells that $N(\mathcal{U}_{4,3})$ is homotopy equivalent to a $2$-torus, so $E^{2}_{1,0}[\Gamma_{3}](4)=H_{1}\big(N(\mathcal{U}_{4,3})\big)\cong \Z^{2}$. 
Thus, we have found a generating set for $H_{1}\big(F_{\bullet}(\Gamma)_{k}\big)$ of the desired degree; next, we show that no smaller generating set exists.

Let $W$ be some other FB-module generating $H_{1}\big(F_{\bullet}(\Gamma_{k})\big)$ such that $W_{m}=0$ for $m\ge N_{k}$.
It follows that the image of $\bigoplus_{m<N_{k}}M^{\text{FI}_{k,o}}\big(E^{2}_{1,0}[\Gamma_{k}](m)\big)$ and $M^{\text{FI}_{k,o}}(W)$ in $H_{1}\big(F_{N_{k}-1}(\Gamma_{k})\big)$ are equal, being all of $H_{1}\big(F_{N_{k}-1}(\Gamma_{k})\big)$.
From the FI$_{k,o}$-module structure of $H_{1}\big(F_{\bullet}(\Gamma_{k})\big)$, it follows that the image of $M^{\text{FI}_{k,o}}(W)_{N_{k}}$ in $H_{1}\big(F_{N_{k}}(\Gamma_{k})\big)$ is precisely all the classes in $H_{1}\big(F_{N_{k}}(\Gamma_{k})\big)$ that arise from linear combinations of classes in $H_{1}\big(F_{N_{k}}(\Gamma_{k})\big)$ that arise by adding a point at the end of an edge of $\Gamma_{k}$ to a class in $H_{1}\big(F_{N_{k}-1}(\Gamma_{k})\big)$, i.e., the classes of the $E^{1}_{0,1}[k](N_{k})$-entry of the Mayer--Vietoris spectral sequence.
Since $E^{2}_{1,0}[\Gamma_{k}](N_{k})\neq0$, it follows that these classes do not generate all of $H_{1}\big(F_{N_{k}}(\Gamma_{k})\big)$, so $W$ is not a generating set, a contradiction.
Therefore, $E^{2}_{1,0}[\Gamma_{k}]$ is a generating set of minimal degree, completing the proof.
\end{proof}

We have shown that $H_{i}\big(F_{\bullet}(\Gamma_{k})\big)$ has the structure of a finitely generated FI$_{k,o}$-module and determined its generation degree as such, proving that it satisfies a notion of representation stability.
The case of general graphs is significantly more complicated. 
When $\Gamma$ has edges not connected to leaves, the operation of adding particles at the leaves is in general insufficient to generate all classes in homology.
Moreover, a direct analogue of the edge stabilization maps of \cite{an2020edge} is not available due to the impossibility of such a map existing at the level of configuration space, though L\"{u}tgehetmann--Recio-Mitter introduce a different potential stabilization map in \cite[Section 5]{lutgehetmann2021topological}.

In the next section we consider the presentation degree of $H_{i}\big(F_{\bullet}(\Gamma_{k})\big)$ as an FI$_{k,o}$-module, by studying the $E^{2}_{2,0}[\Gamma_{k}](n+1)$-entry of our Mayer--Vietoris spectral sequence.

\section{Finite Presentability}\label{presentability}
Having determined the generation degree of $H_{1}\big(F_{\bullet}(\Gamma_{k})\big)$ as an FI$_{k,o}$-module, we turn to bounding its presentation degree.
Doing so gives an upper bound on the number of particles needed to determine when new relations no longer occur in $H_{1}\big(F_{\bullet}(\Gamma_{k})\big)$.

To approach this problem we again rely on the Mayer--Vietoris spectral sequence to determine when all relations in $H_{1}\big(F_{n+1}(\Gamma_{k})\big)$ arise from relations in $H_{1}\big(F_{n}(\Gamma_{k})\big)$.
We will show that it suffices to prove that the $E^{2}_{2,0}[\Gamma_{k}](n+1)$-entry of Mayer--Vietoris spectral sequence for $F_{n+1}(\Gamma_{k})$ eventually vanishes for all $k\ge 4$, something that is not true for $k=3$.
In fact, for all $n$ at least $3$ there are new relations $H_{1}\big(F_{n+1}(\Gamma_{3})\big)$, proving that the problem of finding a finite universal presentation for the homology of the ordered configuration space of graphs is impossible. 

Recall that $E^{2}_{0,2}[\Gamma_{k}](n+1)=H_{2}\big(N(\mathcal{U}_{n+1, k})\big)$. 
Calculating this homology group directly is a challenge; however, we will find an covering of the $3$-skeleton of $N(\mathcal{U}_{n+1, k})$ that will allow us to use the Mayer--Vietoris spectral sequence to determine when $H_{2}\big(N(\mathcal{U}_{n+1, k})\big)$ vanishes.

Consider the $0$-simplices in $N(\mathcal{U}_{n+1, k})$ that arise by fixing an edge $j$ of $\Gamma_{k}$ and letting the label of the particle fixed on this edge vary through all $n+1$ possible labels.
For each of the $n+1$ such $0$-simplices $U_{i,j}$ consider its closed star $S_{i}$. 
As long as $k\ge 4$ and $n\ge 4$ every $3$-dimensional subsimplex of $N(\mathcal{U}_{n+1, k})$ is contained in the star of one of these $0$-simplices, so taking a $\epsilon$-neighborhood of these stars for $\epsilon>0$ sufficiently small will give an open cover $\mathcal{S}_{n+1, k}$ of a subcomplex of $N(\mathcal{U}_{n+1, k})$ containing $N(\mathcal{U}_{n+1, k})^{(3)}$ that we can use to analyze the second homology of $N(\mathcal{U}_{n+1, k})$.

\begin{prop}\label{stars contain everything}
For $k\ge 4$ and $n\ge 4$, let $\tau\subset N(\mathcal{U}_{n+1, k})$ be a subsimplex of dimension at most $3$.
Then $\tau\subset S_{i}$ for some $1\le i\le n+1$.
\end{prop}

\begin{proof}
First, note that every $0$-, $1$-, and $2$-dimensional subsimplex of $N(\mathcal{U}_{n+1, k})$ is contained in some $3$-dimension subsimplex of $N(\mathcal{U}_{n+1, k})$.
This follows from the fact that there are at least $4$ edges in $\Gamma_{k}$ and there are at least $4$ distinct particles in $F_{n+1}(\Gamma_{k})$.
As such, it suffices to show that every $3$-dimensional subsimplex of $N(\mathcal{U}_{n+1, k})$ is in some $S_{i}$.
Let $\tau$ be such a subsimplex, and let $U_{i_{0}, j_{0}}$, $U_{i_{1}, j_{1}}$, $U_{i_{2}, j_{2}}$, and $U_{i_{3}, j_{3}}$ be its vertices---note that $j_{0}, j_{1}, j_{2}$, and $j_{3}$ are distinct and so are $i_{0}, i_{1}, i_{2}$, and $i_{3}$.
If one of these $j_{l}$ is equal to the fixed edge $j$, then $\tau$ is in $S_{i_{l}}$.
Otherwise none of $j_{0}, j_{1}, j_{2}$, and $j_{3}$ are equal to $j$.
In this case, $k\ge 5$ and there is some $U_{i_{4}, j_{4}}$ in $N(\mathcal{U}_{n+1, k})$ such that $j_{4}=j$ and $i_{4}\in \{1, \dots, n+1\}$ is not one of $i_{0}, i_{1}, i_{2}$, or $i_{3}$.
Therefore, $\tau$ is in $S_{i_{4}}$, proving that every $3$-dimensional subsimplex of $N(\mathcal{U}_{n+1, k})$ is in some $S_{i}$ for some $1\le i\le n+1$.
\end{proof}

The above argument can be adapted to show that if $n+1\ge k$, then every subsimplex $\tau\subset N(\mathcal{U}_{n+1, k})$ is in some $S_{i}$, though we do not need the full strength of this statement.

Giving a simplex the standard Euclidean metric induces a metric on the simplicial complex $N(\mathcal{U}_{n+1, k})$, and if we take the $\epsilon$-neighborhood of $S_{i}$ for a sufficiently small $\epsilon>0$, the resulting open neighborhoods $S'_{i}$ have the same intersection pattern as the $S_{i}$.
While $\mathcal{S}_{n+1, k}:=\{S_{i}\}_{i=1}^{n+1}$ might not cover $N(\mathcal{U}_{n+1, k})$, it does cover the latter's $3$-skeleton if $k\ge 4$ and $n\ge 4$.
Since the $3$-skeleton of any CW complex has the same zeroth, first, and second homologies as the entire complex we can use $\mathcal{S}_{n+1, k}$ and the corresponding Mayer--Vietoris spectral sequence to calculate $H_{2}\big(N(\mathcal{U}_{n+1, k})\big)$.
To do this we begin by noting the following fact about the intersection of the $S_{i}$, and hence the $S'_{i}$.

\begin{prop}\label{intersection of stars}
For $l\ge 1$ and $i_{0}, \dots, i_{l}$ distinct, the subspace $S_{i_{0}}\cap\cdots\cap S_{i_{l}}$ of $N(\mathcal{U}_{n+1, k})$ is homotopy equivalent to $N(\mathcal{U}_{n-l, k-1})$.
\end{prop}

\begin{proof}
A $0$-simplex of $N(\mathcal{U}_{n+1, k})$ is in $S_{i_{0}}\cap\cdots\cap S_{i_{l}}$ if and only if it corresponds to a particle with label not in $\{i_{0}, \dots, i_{l}\}$ and not on the fixed edge $j$. 
Since the $S_{i}$ are closed stars of a $0$-simplex of $N(\mathcal{U}_{n+1, k})$, the intersection $S_{i_{0}}\cap\cdots\cap S_{i_{l}}$ is a full subcomplex of $N(\mathcal{U}_{n+1, k})$.
As $N(\mathcal{U}_{n+1, k})$ is flag, it follows that $S_{i_{0}}\cap\cdots\cap S_{i_{l}}$ is homotopy equivalent to $N(\mathcal{U}_{n-l, k-1})$ by the map that sends the $k-1$ other edges of $\Gamma_{k}$ to the edges of $\Gamma_{k-1}$ and the $n-l$ particles not in $\{i_{0}, \dots, i_{l}\}$ to the $n-l$ particles in $F_{n-l}(\Gamma_{k-1})$.
\end{proof}

Given that we wish to show that $H_{2}\big(N(\mathcal{U}_{n+1, k})\big)=0$ and only the $E^{\infty}_{0,2}$-, $E^{\infty}_{1,1}$-, and $E^{\infty}_{p,0}$-entries of the Mayer--Vietoris spectral sequence arising from the cover $\mathcal{S}_{n+1,k}$ of $N(\mathcal{U}_{n+1, k})^{(3)}$ can contribute to this, we wish to show that the corresponding $E^{2}$-page entries are $0$.
Note that the $S_{i}$ are connected, being stars of a $0$-simplex, and that $N(\mathcal{U}_{n-l, k-1})$ is connected for $k\ge 4$ and $n-l\ge 2$.
It follows from Proposition \ref{E2 of mayer--vietoris} that the $E^{2}_{p,0}$-entry of this Mayer--Vietoris spectral sequence is the $p^{\text{th}}$-homology of $N(\mathcal{S}_{n+1, k})$.
With this in mind, we prove the following:

\begin{prop}\label{nerve is a ball}
For $k\ge 4$ and $n\ge 4$, the first and second homology groups of $N(\mathcal{S}_{n+1, k})$ are $0$ .
\end{prop}

\begin{proof}
There is $0$-simplex in $N(\mathcal{S}_{n+1, k})$ for each set $S_{i}$, i.e., there are $n+1$ such $0$-simplices.
When $n\ge 5$, any collection of $2\le l\le 4$ of these $S_{i}$ has non-trivial intersection as it is homotopy equivalent to $N(\mathcal{U}_{k-1, n+1-l})$, which is non-empty and connected as $k\ge 4$.
Thus, any $l\le 4$ of the $S_{i}$ define an $(l-1)$-simplex whose boundary is a signed sum of the $(l-2)$-simplices that arise by forgetting one of the $S_{i}$.
It follows that the $3$-skeleton of $N(\mathcal{S}_{n+1, k})$ is the $3$-skeleton of the $n$-simplex.
For $n\ge3,$ the $3$-skeleton of the $n$-simplex is $2$-connected, so the Hurewicz Theorem implies $H_{1}\big(N(\mathcal{S}_{n+1, k})\big)=H_{2}\big(N(\mathcal{S}_{n+1, k})\big)=0$.

When $n=4$, the we have a similar result, though in this case the $3$-skeleton of $N(\mathcal{S}_{n+1, k})$ is not the $3$-skeleton of the $4$-simplex. 
Instead, it is the $2$-skeleton of the $4$-simplex with with $k-1$ copies of each $3$-simplex of the $3$-skeleton identified at their faces.
This space is still $2$-connected, so $H_{1}\big(N(\mathcal{S}_{n+1, k})\big)=H_{2}\big(N(\mathcal{S}_{n+1, k})\big)=0$.
\end{proof}

We use these results to determine the second homology of $N(\mathcal{U}_{n+1, k})$ for $n$ sufficiently large with respect to $k$.
The vanishing of this homology group corresponds to the non-existence of new relations in $H_{1}\big(F_{n+1}(\Gamma_{k})\big)$.

\begin{lem}\label{second homology of nerve}
For $k\ge 7$ and $n\ge 3$, $k=6$ and $n\ge 4$, $k=5$ and $n\ge 5$, and $k=4$ and $n\ge 6$, the second homology of the nerve of the cover $\mathcal{U}_{n+1, k}$ of $F_{n+1}(\Gamma_{k})$ is $0$, i.e., $H_{2}\big(N(\mathcal{U}_{n+1, k})\big)=0$.
Moreover, if $k\ge 7$, then $H_{2}\big(N(\mathcal{U}_{k, 3})\big)\neq 0$.
\end{lem}

\begin{proof}
The $E^{1}_{p,q}$-entry of the Mayer--Vietoris spectral sequence with respect to the covering $\mathcal{S}_{n+1,k}$ of $N(\mathcal{U}_{n+1, k})$ is given by $\bigoplus_{i_{0}, \dots, i_{p}}H_{q}(S'_{i_{0}}\cap \cdots\cap S'_{i_{p}})$.
By Proposition \ref{intersection of stars}, the intersection $S'_{i_{0}}\cap \cdots\cap S'_{i_{p}}\simeq S_{i_{0}}\cap \cdots\cap S_{i_{p}}$ is homotopy equivalent to $N(\mathcal{U}_{k-1, n-p})$ for $p\ge 1$ and $k\ge 4$.
Moreover, $H_{q}(S_{i_{0}})=0$ for all $q\ge 1$, as $S_{i_{0}}$ is contractible onto the $0$-simplex where particle $i_{0}$ sits on the chosen fixed edge. 
See Figure \ref{MayerVietoris E1 of nerve}.

\begin{figure}[h]
\centering
\begin{tikzpicture} \footnotesize
  \matrix (m) [matrix of math nodes, nodes in empty cells, nodes={minimum width=3ex, minimum height=5ex, outer sep=2ex}, column sep=3ex, row sep=3ex]{
 3    &  0&\bigoplus_{\{i_{0}, i_{1}\}} H_{3}\big(N(\mathcal{U}_{n-1, k-1})\big) &  \bigoplus_{\{i_{0}, i_{1}, i_{2}\}} H_{3}\big(N(\mathcal{U}_{n-2, k-1})\big)  & \bigoplus_{\{i_{0}, \dots, i_{3}\}} H_{3}\big(N(\mathcal{U}_{n-3, k-1})\big)& \\  
 2    &  0 &\bigoplus_{\{i_{0}, i_{1}\}} H_{2}\big(N(\mathcal{U}_{n-1, k-1})\big) &  \bigoplus_{\{i_{0}, i_{1}, i_{2}\}} H_{2}\big(N(\mathcal{U}_{n-2, k-1})\big)  & \bigoplus_{\{i_{0}, \dots, i_{3}\}} H_{2}\big(N(\mathcal{U}_{n-3, k-1})\big)&  \\          
 1    &  0  &\bigoplus_{\{i_{0}, i_{1}\}} H_{1}\big(N(\mathcal{U}_{n-1, k-1})\big) &  \bigoplus_{\{i_{0}, i_{1}, i_{2}\}} H_{1}\big(N(\mathcal{U}_{n-2, k-1})\big)  & \bigoplus_{\{i_{0}, \dots, i_{3}\}} H_{1}\big(N(\mathcal{U}_{n-3, k-1})\big)& \\             
  0    &  \bigoplus_{\{i_{0}\}} H_{0}(S_{i_{0}})  &\bigoplus_{\{i_{0}, i_{1}\}} H_{0}\big(N(\mathcal{U}_{n-1, k-1})\big) &  \bigoplus_{\{i_{0}, i_{1}, i_{2}\}} H_{0}\big(N(\mathcal{U}_{n-2, k-1})\big)  & \bigoplus_{\{i_{0}, \dots, i_{3}\}} H_{0}\big(N(\mathcal{U}_{n-3, k-1})\big)& \\       
 \quad\strut &  0  &  1  & 2  &3&\\}; 

\draw[thick] (m-1-1.east) -- (m-5-1.east) ;
\draw[thick] (m-5-1.north) -- (m-5-5.north east) ;
\end{tikzpicture}
\caption{The $E^{1}$-page of the Mayer--Vietoris spectral sequence for the cover $\mathcal{S}_{n+1,k}$ of $N(\mathcal{U}_{n+1, k})$.}
\label{MayerVietoris E1 of nerve}
\end{figure}

For $k\ge 4$, Proposition \ref{nerve is a ball} proves that the $E^{2}_{p,0}$-entry of this spectral sequence is $0$ for $p=1,2$.
Additionally, Lemma \ref{contract 1-cycles} proves that if $k=3$ and $n\ge 4$, or $k=4$ and $n\ge 3$, or $k\ge 5$ and $n\ge 2$ that $H_{1}\big(N(\mathcal{U}_{n+1, k})\big)=0$.
Therefore, if $k=4$ and $n\ge 6$, or $k=5$ and $n\ge 5$, or $k\ge 6$ and $n\ge 4$ the $E^{1}_{1,1}$- and hence $E^{2}_{1,1}$-entries of this spectral sequence are $0$.
It follows that in these situations, the $E^{2}$-page of the Mayer--Vietoris spectral sequence is of the form in Figure \ref{MayerVietoris E2 of nerve}.
This implies $E^{2}_{0,2}=E^{\infty}_{0,2}=0$, $E^{2}_{1,1}=E^{\infty}_{1,1}=0$, and $E^{2}_{2,0}=E^{\infty}_{2,0}=0$.
Since this spectral sequence converges to the homology of a subcomplex of $N(\mathcal{U}_{n+1, k})$ containing $N(\mathcal{U}_{n+1, k})^{(3)}$, it follows that $H_{2}\big(N(\mathcal{U}_{n+1, k})\big)=0$ for $k=4$ and $n\ge 6$, $k=5$ and $n\ge 5$, and $k\ge 6$ and $n \ge 4$.

To show that $H_{2}\big(N(\mathcal{U}_{n+1, k})\big)=0$ for $k\ge 7$ and $n=3$, note that interchanging the roles of the leaves and the particles yields a homotopy equivalence between $N(\mathcal{U}_{n+1, k})$ and $N(\mathcal{U}_{n+1, k})$.
Since $H_{2}\big(N(\mathcal{U}_{4, k})\big)=0$ for $k\ge 7$ by the above, it follows that $H_{2}\big(N(\mathcal{U}_{k, 4})\big)=0$ for $k\ge 7$.

This homotopy equivalence also proves that $H_{2}\big(N(\mathcal{U}_{k, 3})\big)=H_{2}\big(N(\mathcal{U}_{3, k})\big)$.
Since $N(\mathcal{U}_{3, k})$ is a connected simplicial complex and Lemma \ref{contract 1-cycles} proves that $H_{1}\big(N(\mathcal{U}_{3, k})\big)=0$ for $k\ge 5$, an Euler characteristic computation can be used to show that $H_{2}\big(N(\mathcal{U}_{3, k})\big)\cong \Z^{k^{3}-6k^{2}+8k-1}$ for $k\ge 5$.
It follows that if $k\ge 7$, then $H_{2}\big(N(\mathcal{U}_{k, 3})\big)\neq 0$.

\begin{figure}[h]
\centering
\begin{tikzpicture} \footnotesize
  \matrix (m) [matrix of math nodes, nodes in empty cells, nodes={minimum width=3ex, minimum height=5ex, outer sep=2ex}, column sep=3ex, row sep=3ex]{
 3    &  0&* &  *  & *& \\  
 2    &  0 &* &  *  & *&  \\          
 1    &  0  &0 & * & *& \\             
  0    &  \Z  &0 &  0  & *& \\       
 \quad\strut &  0  &  1  & 2  &3&\\}; 

\draw[thick] (m-1-1.east) -- (m-5-1.east) ;
\draw[thick] (m-5-1.north) -- (m-5-5.north east) ;
\end{tikzpicture}
\caption{The $E^{2}$-page of the Mayer--Vietoris spectral sequence for the cover $\mathcal{S}_{n+1,k}$ of $N(\mathcal{U}_{n, k+1})$ when $n$ is sufficiently large with respect to $k$.}
\label{MayerVietoris E2 of nerve}
\end{figure}
\end{proof}

\begin{remark}\label{wedge of spheres}
One can improve the result of Proposition \ref{nerve is a ball} by using the technique of Lemma \ref{second homology of nerve} to prove that if $n$ is large with respect to $k$, then $N(\mathcal{U}_{n+1, k})$ only has homology in degree $k-1$.
It is unclear what extra information this tells us about $H_{1}\big(F_{\bullet}(\Gamma_{k})\big)$, though the author believes that this is connected to the minimal length of a resolution of $H_{1}\big(F_{\bullet}(\Gamma_{k})\big)$ by free FI$_{k, o}$-modules.
\end{remark}

With this in mind we recall the statement of Theorem \ref{no new relations}, which we now prove.

\begin{T2}
  \thmtextone
\end{T2} 

\begin{proof}
The first page differential $d^{1}:E^{1}_{1,1}[\Gamma_{k}](n+1)\to E^{1}_{0,1}[\Gamma_{k}](n+1)$ of the Mayer--Vietoris spectral sequence for $F_{n+1}(\Gamma_{k})$ with respect to the cover $\mathcal{U}_{n+1,k}$ does not introduce new relations in homology.
Rather, this differential merely identifies classes in $H_{1}\big(F_{n+1}(\Gamma_{k})\big)$ where two different particles are fixed on two different edges of $\Gamma$, i.e., if $\alpha\in H_{1}\big(F_{n-1}(\Gamma_{k})\big)$ is a class on $n-1$ particles, then $\alpha$ induces an element $\tilde{\alpha}$ in $E^{1}_{1,1}[\Gamma_{k}](n+1)$ where we fix two additional particles $i_{0}$ and $i_{1}$ at different leaves of $\Gamma$.
It follows that $d^{1}(\tilde{\alpha})$ is the difference of classes in distinct copies of $H_{1}\big(F_{n}(\Gamma_{k})\big)\subset E^{1}_{0,1}[\Gamma_{k}](n+1)$ where there is an uncounted particle $i_{0}$ is fixed on a leaf $j_{0}$ in one copy of $F_{n}(\Gamma_{k})$ and a different uncounted particle $i_{1}$ is fixed on a different leaf $j_{1}$ in different copy of $F_{n}(\Gamma_{k})$.
It follows that this differential does not introduce relations in $H_{1}\big(F_{n+1}(\Gamma_{k})\big)$, so the only differential that could is $d^{2}:E^{2}_{2, 0}[\Gamma_{k}](n+1)\to E^{2}_{0,1}[\Gamma_{k}](n+1)$. 
Therefore, if we can show that $E^{2}_{0,1}[\Gamma_{k}](n+1)=0$, then all relations in $H_{1}\big(F_{n+1}(\Gamma_{k})\big)$ would arise from relations already described in the $E^{1}_{0,1}[\Gamma_{k}](n+1)$-entry, i.e., every relation in $H_{1}\big(F_{n+1}(\Gamma_{k})\big)$ can be described as a sum of relations among classes in $H_{1}\big(F_{n}(\Gamma_{k})\big)$ where an extra particle has been fixed at the end of the same edge in all of the involved classes.

By Lemma \ref{second homology of nerve}, the $E^{2}_{2, 0}[\Gamma_{k}](n+1)$-entry of the Mayer--Vietoris spectral sequence for $F_{n+1}(\Gamma_{k})$ is $0$ for $k=4$ and $n\ge 6$, for $k=5$ and $n\ge 5$, for $k=6$ and $n\ge 4$, and for $k\ge 7$ and $n\ge 3$.
Therefore, the image of the $d^{2}$-differential from this entry into $E^{2}_{0,1}[\Gamma_{k}](n+1)$ is $0$.

By Theorem \ref{rep stability for star}, we have that $H_{1}\big(F_{\bullet}(\Gamma_{k})\big)$ is finitely generated as an FI$_{k,o}$-module in degree $2$ for $k\ge 5$ and in degree $3$ for $k=4$.
Since $E^{2}_{2,0}[\Gamma_{k}](n+1)$ is $0$ for $n\ge 6$ and $k=4$, for $n\ge 5$ and $k=5$, for $n\ge 4$ and $k=6$, and for $n\ge 3$ and $k\ge 7$, the FI$_{k,o}$-module of relations is generated at most these degrees.
Moreover, $E^{2}_{2,0}[\Gamma_{k}](n+1)$ is finite dimensional being the second homology of finite cellular complex, so the FI$_{k,o}$ module of relations is also finitely generated.

To confirm that $H_{1}\big(F_{\bullet}(\Gamma_{k})\big)$ cannot be presented in degree less than $3$ for $k\ge 7$, note that $H_{1}\big(F_{1}(\Gamma_{k})\big)=0$ for all $k$ and $H_{1}\big(F_{2}(\Gamma_{k})\big)\neq 0$ for all $k\ge 3$.
For $k\ge 4$ there must relations among the images of the classes of $H_{1}\big(F_{2}(\Gamma_{k})\big)$ in $H_{1}\big(F_{3}(\Gamma_{k})\big)$ under the insertion maps as $E^{2}_{2, 0}[\Gamma_{k}](3)\neq 0$, so $H_{1}\big(F_{\bullet}(\Gamma_{k})\big)$ cannot be presented in degree less than $3$.
Therefore, for $k\ge 4$, the sequence of homology groups $H_{1}\big(F_{\bullet}(\Gamma_{k})\big)$ has the structure of a finitely presented FI$_{k,o}$-module, presented in degree at most $6$ for $k=4$, degree at most $5$ for $k=5$, degree at most $4$ for $k=6$,  and degree $3$ for $k\ge 7$.
\end{proof}

It follows that every class in $H_{1}\big(F_{n+1}(\Gamma_{k})\big)$ can be generated by classes on at most $4$ points, and that for $k\ge 4$ every relation in homology arises from relations on classes on at most $6$ points.

\begin{remark}
The proof of Theorem \ref{no new relations} proves that $H_{1}\big(F_{\bullet}(\Gamma_{k})\big)$ cannot be presented in degree $2$ for $k\ge 7$.
Similar arguments can be used to show that $H_{1}\big(F_{\bullet}(\Gamma_{4})\big)$ cannot be presented in degree $6$ and $H_{1}\big(F_{\bullet}(\Gamma_{6})\big)$ cannot be presented in degree $4$.
These arguments rely on the fact that $H_{2}\big(N(\mathcal{U}_{6,4})\big)\cong H_{2}\big(N(\mathcal{U}_{4,6})\big)\neq0$, which can be determined by calculating that $H_{2}\big(N(\mathcal{U}_{6,4});\Z/2\Z\big)\cong H_{2}\big(N(\mathcal{U}_{4,6});\Z/2\Z\big)\cong (\Z/2\Z)^{5}$ and applying the Universal Coefficients Theorem for homology.
Similar calculations show that $H_{2}\big(N(\mathcal{U}_{5,5});\Z/2\Z\big)=0$; since there is no reason to believe that this complex has torsion, this would imply that $H_{1}\big(F_{\bullet}(\Gamma_{5})\big)$ can be presented in degree $4$.
\footnote{These calculations were done in Macaulay2. For $k=4$ and $n+1=6$ and for $k=5$ and $n=5$, this was done with the SimplicialComplexes package by labeling the top dimensional simplices of $N(\mathcal{U}_{n+1,k})$ by $k$-tuples of elements, where the $(l+1)^{\text{th}}$-element of each $k$-tuple is drawn from the set $\{1+l(n+1), \dots, n+1+l(n+1)\}$, and all of the elements are distinct modulo $n+1$. The isomorphism between $N(\mathcal{U}_{6,4})$ and $N(\mathcal{U}_{4,6})$ that arises from interchanging the roles of the edges of $\Gamma$ and the particles, yields the $N(\mathcal{U}_{4,6})$ result.}
\end{remark}

In the next section we show that $H_{1}\big(F_{\bullet}(\Gamma_{3})\big)$ is not finitely presented as an FI$_{3,o}$-module.

\subsection{Relations for $k=3$}

In this subsection we discuss relations in $H_{1}\big(F_{\bullet}(\Gamma_{3})\big)$.
Unlike larger $k$, the $E^{2}_{2,0}[\Gamma_{3}](n+1)$-entry of the Mayer--Vietoris spectral sequence for $F_{n+1}(\Gamma_{3})$ is never $0$ for large $n$.
In fact, using Lemma \ref{contract 1-cycles}, the torsion-freeness of $H_{1}\big(F_{n}(\Gamma_{3})\big)$, and an Euler characteristic computation, one can check that $E^{2}_{2,0}[\Gamma_{3}](n+1)\cong \Z^{n^{3}-3n^{2}-n+2}$.
As such, if we limit ourselves to the minimal generators suggested by our spectral sequence, i.e., the generators arising from the nonzero $E^{2}_{1,0}[\Gamma_{3}](n+1)$-entry, we always need new relations on classes of $n+1$ particles that do not arise from relations on classes of $n$ particles.
This raises the question of whether $H_{1}\big(F_{\bullet}(\Gamma_{3})\big)$ is finitely presentable as an FI$_{3,o}$-module.
It could be the case that we have merely chosen a bad generating set and a different finite generating set would lead to a finite presentation.
We will prove that this is not the case, and that $H_{1}\big(F_{\bullet}(\Gamma_{3})\big)$ is not finitely presentable as an FI$_{3,o}$-module.
This is in stark contrast to the homology groups of the ordered configuration space of the other star graphs, which are finitely presentable as FI$_{k, o}$-modules, the homology groups of the unordered configuration space of a graph, which are finitely presentable as modules over $\Z[E]$ \cite[Theorem 1.1]{an2020edge}, and the homology groups of the ordered configuration space of a connected non-compact finite type manifold of dimension at least 2, which are are free as FI-modules \cite[Theorem 6.4.3]{church2015fi}.

We begin by proving that if we restrict ourselves to the generating set described by the $E^{1}_{0,1}[\Gamma_{3}]$-entry of the Mayer--Vietoris spectral sequence, then the kernel of map from the resulting free FI$_{3,o}$-module to $H_{1}\big(F_{\bullet}(\Gamma_{3})\big)$ is not finitely generated.

\begin{prop}\label{kernel is not finitely generated}
Let $E^{2}_{1,0}[\Gamma_{3}]$ denote the FB-module such that $E^{2}_{1,0}[\Gamma_{3}]_{n+1}$ corresponds to the homology generators of $H_{1}\big(F_{n+1}(\Gamma_{3})\big)$ given by the $E^{2}_{1,0}[\Gamma_{3}](n+1)$-entries of the Mayer--Vietoris spectral sequence.
The kernel of $M^{\text{FI}_{3,o}}\big(E^{2}_{1,0}[\Gamma_{3}]\big)\twoheadrightarrow H_{1}\big(F_{\bullet}(\Gamma_{3})\big)$ is not finitely generated.
\end{prop}

\begin{proof}

Let $K$ denote the kernel of the obvious map $M^{\text{FI}_{3,o}}\big(E^{2}_{1,0}[\Gamma_{3}]\big)\twoheadrightarrow H_{1}\big(F_{\bullet}(\Gamma_{3})\big)$, and let $E^{2}_{2,0}[\Gamma_{3}]$ denote the FB-module whose degree $(n+1)$-term is $E^{2}_{2,0}[\Gamma_{3}](n+1)$.
By the construction of the Mayer--Vietoris spectral  we have the following commutative diagram

\begin{center}
\begin{tikzcd}
            & {M^{\text{FI}_{3,o}}\big(E^{2}_{2,0}[\Gamma_{3}]\big)} \arrow[r] \arrow[d, two heads] & {M^{\text{FI}_{3,o}}\big(E^{2}_{1,0}[\Gamma_{3}]\big)} \arrow[r] \arrow[d] & H_{1}\big(F_{\bullet}(\Gamma_{3})\big) \arrow[r] \arrow[d] & 0 \\
0 \arrow[r] & K \arrow[r]                                  & {M^{\text{FI}_{3,o}}\big(E^{2}_{1,0}[\Gamma_{3}]\big)} \arrow[r]           & H_{1}\big(F_{\bullet}(\Gamma_{3})\big) \arrow[r]           & 0
\end{tikzcd}
\end{center}
where the right two vertical arrows are the identity, and the left vertical arrow is a surjection.

Let $V$ be an FB-module such that for some $N$ we have $V_{N}\neq 0$ and $V_{n+1}=0$ for all $n\ge N$ and such that the diagram

\begin{center}
\begin{tikzcd}
            & {M^{\text{FI}_{3,o}}(V)} \arrow[r] \arrow[d, two heads] & {M^{\text{FI}_{3,o}}\big(E^{2}_{1,0}[\Gamma_{3}]\big)} \arrow[r] \arrow[d] & H_{1}\big(F_{\bullet}(\Gamma_{3})\big) \arrow[r] \arrow[d] & 0 \\
0 \arrow[r] & K \arrow[r]                                  & {M^{\text{FI}_{3,o}}\big(E^{2}_{1,0}[\Gamma_{3}]\big)} \arrow[r]           & H_{1}\big(F_{\bullet}(\Gamma_{3})\big) \arrow[r]           & 0
\end{tikzcd}
\end{center}
commutes, i.e., $V$ is a generating set for the kernel that is non-zero in only finitely many degrees.

The images of $M^{\text{FI}_{3,o}}(V)_{N}$ and $\bigoplus_{m\le N}M^{\text{FI}_{3,o}}\big(E^{2}_{2,0}[\Gamma_{3}](m)\big)_{N}$ in $K_{N}$ are identical, being all of $K_{N}$.
It follows from the FI$_{3,o}$-module structures of $H_{1}\big(F_{\bullet}(\Gamma_{3})\big)$ and $M^{\text{FI}_{3,o}}\big(E^{2}_{1,0}[\Gamma_{3}]\big)$ that $K$ is the FI$_{3, o}$-module of relations in $H_{1}\big(F_{\bullet}(\Gamma_{3})\big)$, i.e., there are maps $K_{n}\to K_{n+1}$ that correspond to sending a relation on classes of $n$ particles to a relation on a class on $n+1$ particles by fixing a new particle on a leaf of $\Gamma_{3}$.
For sufficiently large $n$, we have that $M^{\text{FI}_{3,o}}(V)_{n}=K_{n}$ and $\bigoplus_{m\le N}M^{\text{FI}_{3,o}}\big(E^{2}_{2,0}[\Gamma_{3}](m)\big)_{n}\neq K_{n}$, but $M^{\text{FI}_{3,o}}(V)$ and $\bigoplus_{m\le N}M^{\text{FI}_{3,o}}\big(E^{2}_{2,0}[\Gamma_{3}](m)\big)$ must have the same image in $K_{n}$, since $V$ is generating and they have the same image at the generating degree of $V$.
This is a result of the fact $K_{n}$ that cannot be described by the relations arising from $\bigoplus_{m\le N}M^{\text{FI}_{3,o}}\big(E^{2}_{2,0}[\Gamma_{3}](m)\big)_{n}$ as these relations do not describe those coming from the $E^{2}_{2,0}[\Gamma_{3}](n)$-entry of the spectral sequence, since they are already taken into account in the $E^{2}_{0,1}[\Gamma_{3}](n)$-entry.
Thus, $M^{\text{FI}_{3,o}}(V)$ is not a generating set for $K$, so no finite generating set for the kernel of $M^{\text{FI}_{3,o}}\big(E^{2}_{2,0}[\Gamma_{3}]\big)\twoheadrightarrow H_{1}\big(F_{\bullet}(\Gamma_{3})\big)$ exists.
\end{proof}

\begin{remark}
Something more is probably true.
Calculations suggest that $M^{\text{FI}_{3,o}}\big(E^{2}_{2,0}[\Gamma_{3}]\big)$ \emph{is} the kernel of $M^{\text{FI}_{3,o}}\big(E^{2}_{1,0}[\Gamma_{3}]\big)\twoheadrightarrow H_{1}\big(F_{\bullet}(\Gamma_{3})\big)$, i.e., we have an exact sequence
\[
0\to M^{\text{FI}_{3,o}}\big(E^{2}_{2,0}[\Gamma_{3}]\big)\to M^{\text{FI}_{3,o}}\big(E^{2}_{1,0}[\Gamma_{3}]\big)\to H_{1}\big(F_{\bullet}(\Gamma_{3})\big)\to 0.
\]
The analogous statement for larger star graphs is not true, and for $k\ge 4$ there are relations among the relations arising from the $E^{2}_{2,0}[\Gamma_{k}](n+1)$-entry of the Mayer--Vietoris spectral sequence; see Remark \ref{wedge of spheres}.
This freeness among the relations of $H_{1}\big(F_{\bullet}(\Gamma_{3})\big)$ should correspond to the vanishing of the $E^{2}_{p,0}[\Gamma_{3}](n+1)$-entries of the spectral sequence for $p\ge 3$.
\end{remark}

Next, we prove that our choice of generators does not matter by showing that if we have a finite presentation of $H_{1}\big(F_{\bullet}(\Gamma_{3})\big)$, then every finite generating set is part of a finite presentation.

\begin{prop}\label{all presentations are the same}
Let 
\[
M^{\text{FI}_{3,o}}(V)\to M^{\text{FI}_{3,o}}(W)\to U\to 0
\]
be a finite presentation of an FI$_{3,o}$-module $U$, then given any finitely generated FB-module of generators $W'$ one can find a finitely generated FB-module of relations $V'$ such that
\[
M^{\text{FI}_{3,o}}(V')\to M^{\text{FI}_{3,o}}(W')\to U\to 0.
\]
is a finite presentation of $U$.
\end{prop}

\begin{proof}
We may replace $M^{\text{FI}_{3,o}}(V)$ with $K$, the kernel of $M^{\text{FI}_{3,o}}(W)\twoheadrightarrow U$.
By definition, $K$ is finitely generated by $V$, and $K$ fits into the short exact sequence of FI$_{3,o}$-modules
\[
0\to K\to M^{\text{FI}_{3,o}}(W)\to U\to 0.
\]

Given $W'$ another finitely generated module of generators for $U$, let $K'$ be the kernel of $M^{\text{FI}_{3,o}}(W')\twoheadrightarrow U$, i.e., we have a short exact sequence of FI$_{3,o}$-modules
\[
0\to K'\to M^{\text{FI}_{3,o}}(W')\to U\to 0.
\]

Since $M^{\text{FI}_{3,o}}(W)$ and $M^{\text{FI}_{3,o}}(W')$ are free FI$_{3,o}$-modules we can lift the identity map $U\to U$ to a map $f:M^{\text{FI}_{3,o}}(W)\to M^{\text{FI}_{3,o}}(W')$, and we can restrict this to a map $\tilde{f}:K\to K'$ to make the following commutative diagram

\begin{center}
\begin{tikzcd}
0 \arrow[r] & K \arrow[r] \arrow[d, "\tilde{f}"] & {M^{\text{FI}_{3,o}}(W)} \arrow[r] \arrow[d, "f"] & U \arrow[r] \arrow[d] & 0 \\
0 \arrow[r] & K' \arrow[r]                       & {M^{\text{FI}_{3,o}}(W')} \arrow[r]               & U \arrow[r]                                & 0

\end{tikzcd}
\end{center}
where the right most vertical arrow is the identity.

By the snake lemma, there is an isomorphism of FI$_{3,o}$-modules $\text{Coker}\tilde{f}\cong \text{Coker}f$, so we have a short exact sequence of FI$_{3,o}$-modules
\[
0\to \text{Im}\tilde{f}\to K'\to \text{Coker}f\to 0.
\]
Since $\text{Im}\tilde{f}$ and $\text{Coker}f$ are finitely generated FI$_{3,o}$-modules, $K'$ must be finitely generated as well.
\end{proof}

Since the module of relations for the generating set coming from the Mayer--Vietoris spectral sequence is not finitely generated as it has terms in infinitely many degrees, it follows that $H_{1}\big(F_{\bullet}(\Gamma_{k})\big)$ is not a finitely presentable FI$_{3,o}$-module.

\begin{T3}
  \thmtexttwo
\end{T3}

\begin{proof}
By Proposition \ref{kernel is not finitely generated}, Corollary \ref{E210 is 0}, and Euler characteristic considerations, we have found a presentation for $H_{1}\big(F_{\bullet}(\Gamma_{k})\big)$ with finitely many generators whose kernel cannot be finitely generated. 
Therefore, Proposition \ref{all presentations are the same} proves that we cannot find a finite presentation of $H_{1}\big(F_{\bullet}(\Gamma_{k})\big)$ as an FI$_{3,o}$-module.
\end{proof}

One may interpret Theorem \ref{infinite presentation for 3} as implying that the ordered configuration space of the star graph $\Gamma_{3}$ is significantly more complex than the ordered configuration spaces of the other star graphs.
Moreover, Theorem \ref{infinite presentation for 3} proves that one cannot find a finite universal presentation for the homology of the ordered configuration space of graphs.

\section{Appendix}\label{appendix}

In this auxiliary section we include illustrations of generators for the $E^{2}_{1,0}[\Gamma_{k}]$-terms of the Mayer--Vietoris spectral sequence for $H_{1}\big(F_{\bullet}(\Gamma_{k})\big)$ when $k=3,4$.
These generators are not minimal as there will be relations among them, but they are sufficient.
We hope that these pictures can serve as inspiration for understanding the generators of the homology of the ordered configuration space of more complex graphs.

For more than $2$ particles we take the convention of always having half-stars, see \cite{an2020edge, an2022asymptotic}, as moving counterclockwise.
This is an arbitrary choice that leads to the symmetric group span of the classes depicted being larger than needed. 
The only way to remedy this would be to take averages over all possible orientations of the half-stars when the number of particles is greater than $2$, something that is not possible with integer coefficients.

\subsection{$\Gamma_{3}$}

Sections \ref{rep stab section} and \ref{presentability} show that $\Gamma_{3}$ has the most complex ordered configuration space of all the star graphs.
In our proof of Theorem \ref{rep stability for star}, we found a generator for $H_{1}\big(F_{\bullet}(\Gamma_{3})\big)$ in degree $4$.
A careful examination of the Mayer--Vietoris spectral sequence proves that there are generators in degrees $2$ and $3$ as well.
In this subsection we provide pictures of these generators.

\subsubsection{$n=2$}\label{star class on 2 particles}
When $n=2$, we have that $H_{1}\big(F_{2}(\Gamma_{3})\big)\cong \Z$, which is generated by a single class, a representative of which is depicted in Figure \ref{NewCyclesF2G3}.
The symmetric group $S_{2}$ acts trivially on this class, which is the ordered analogue of the star classes of \cite{an2020edge, an2022asymptotic}.
Unlike its unordered counterparts, the image of the free FI$_{3,o}$-module generated by this class is not all of $H_{1}\big(F_{\bullet}(\Gamma_{3})\big)$.
In fact, one can show that asymptotically that these ordered star classes only account for only half of $H_{1}\big(F_{n}(\Gamma_{3})\big)$.

\begin{figure}[h]
\centering
\captionsetup{width=.8\linewidth}
\includegraphics[width = 6cm]{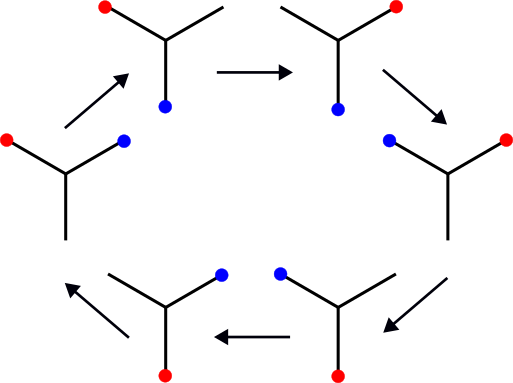}
\caption{The sole class of $H_{1}\big(F_{2}(\Gamma_{3})\big)$ generating $E^{2}_{1,0}[\Gamma_{3}](2)$.
}
\label{NewCyclesF2G3}
\end{figure}

\subsubsection{$n=3$}

When $n=3$, we have that $E^{2}_{1,0}[\Gamma_{3}](3)\cong \Z^{4}$.
By \cite[Theorem 3.3]{lu14} upon taking rational coefficients, this decomposes as two copies of the alternating representation of $S_{3}$ and one copy of the standard representation of $S_{3}$. 
In Figure \ref{NewCycles2F3G3} we demonstrate a class whose orbit, after accounting for our choice of orientation of the half-stars corresponding to the edges of the hexagon, under the $S_{3}$-action arising from the symmetries of the labeling of the particles generates two copies of the alternating representation.
Moreover, any one of the three cycles generated by this action can be written as a sum of the two others.
The non-zero $E^{2}_{2,0}[\Gamma_{3}](4)$-entry of the Mayer--Vietoris spectral sequence corresponds to a relation between the $24$ classes in $H_{1}\big(F_{4}(\Gamma_{3})\big)$ that arise from applying all three insertion maps to the classes in Figure \ref{NewCycles2F3G3}.
In Figure \ref{NewCycles1F3G3} we demonstrate a generator of the standard representation, again needing to factor in for the fact we have made choice to orient our half-stars counterclockwise on each edge of the hexagon.

Note that in the cycle of Figure \ref{NewCycles2F3G3}, in the movements of particles corresponding to sides of the hexagon there is never a particle fixed on the bottom leaf of $\Gamma_{3}$ and all particles are involved in the movements on at least one side of the hexagon.
In the cycle of Figure \ref{NewCycles1F3G3}, in each movement corresponding to a side of the hexagon the blue particle is never fixed, whereas every leaf is involved at least once.
This can be interpreted via the duality between edges and particles; if we consider the $S_{3}$-action arising from the symmetries of $\Gamma_{3}$, the cycle of Figure \ref{NewCycles2F3G3} generates a copy of the alternating representation, and the cycle of Figure \ref{NewCycles1F3G3} generates the standard representation.

\begin{figure}[h]
\centering
\captionsetup{width=.8\linewidth}
\includegraphics[width = 10cm]{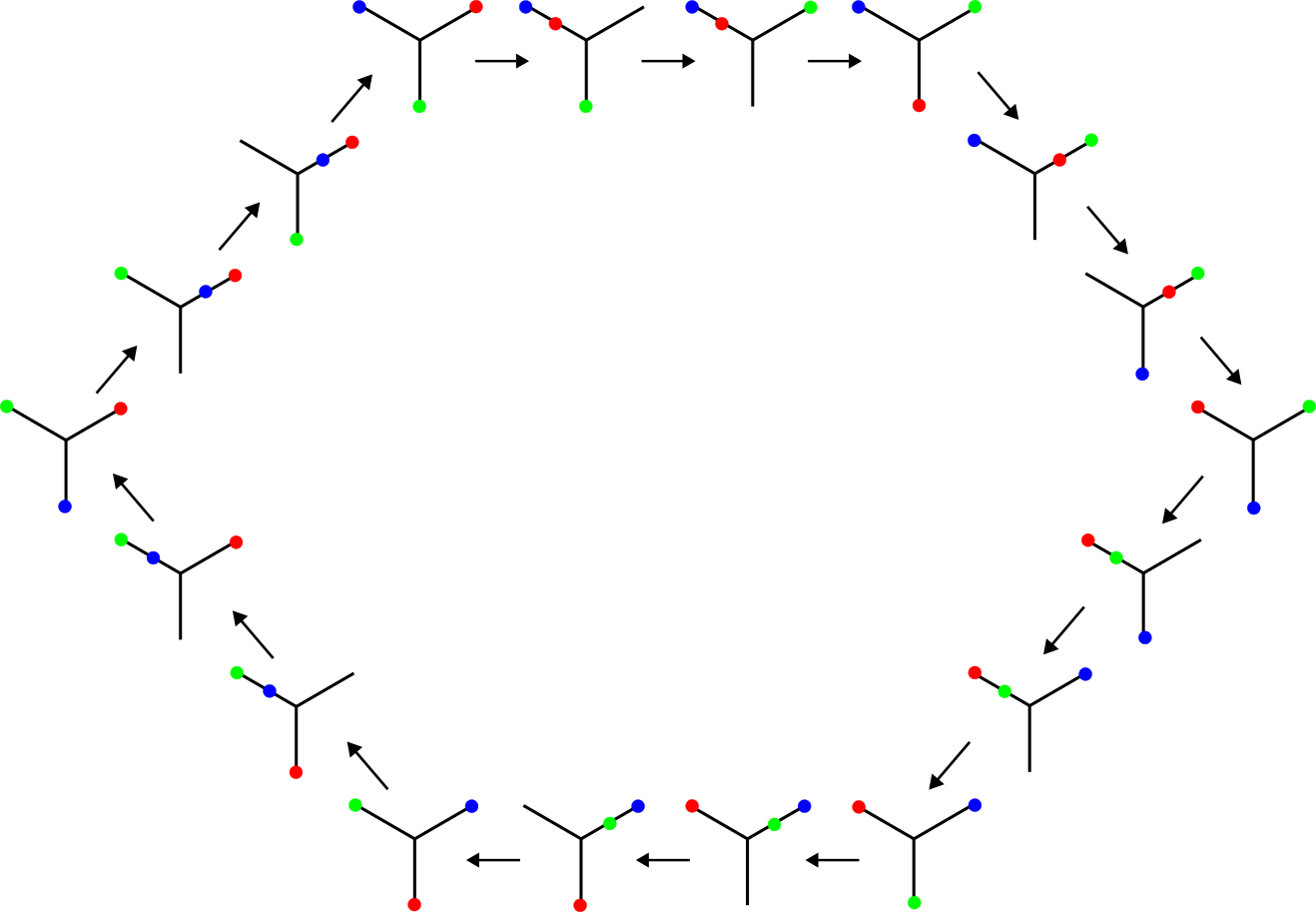}
\caption{A class in $H_{1}\big(F_{3}(\Gamma_{3})\big)$ whose $S_{3}$-orbit generates half of $E^{2}_{1,0}[\Gamma_{3}](4)$, corresponding to the two copies of the alternating representation.
}
\label{NewCycles2F3G3}
\end{figure}

\begin{figure}[H]
\centering
\captionsetup{width=.8\linewidth}
\includegraphics[width = 10cm]{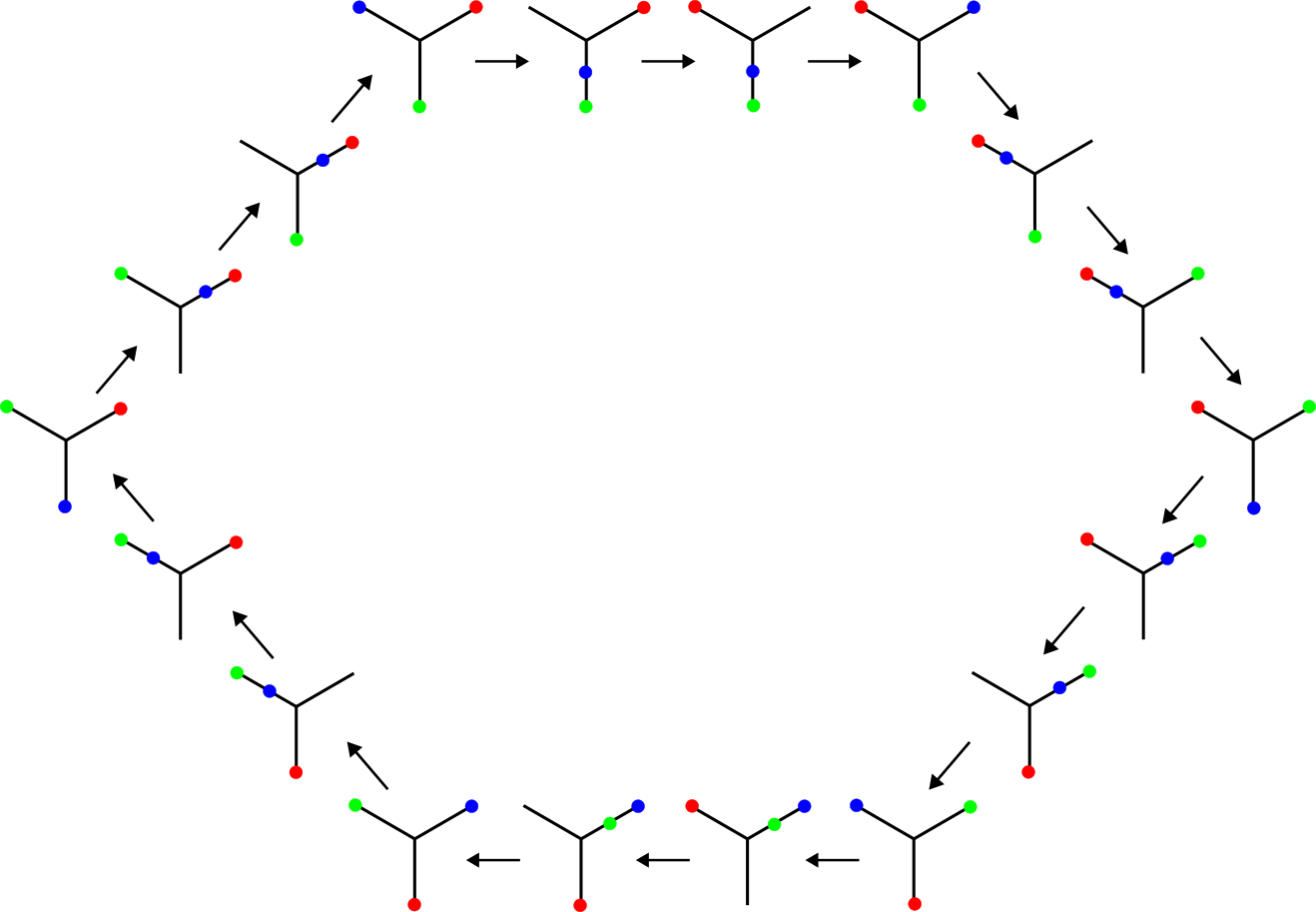}
\caption{A class in $H_{1}\big(F_{3}(\Gamma_{3})\big)$ whose $S_{3}$-orbit generates the other half of $E^{2}_{1,0}[\Gamma_{3}](4)$ corresponding to the standard representation.
}
\label{NewCycles1F3G3}
\end{figure}

\subsubsection{$n=4$}

When $n=4$, we have that $E^{2}_{1, 0}[\Gamma_{3}](4)\cong \Z^{2}$.
We demonstrate two generators for these classes in Figure \ref{NewCyclesF4G3}.
After taking rational coefficients and quotienting by the classes that come from our insertion maps, these classes yield the irreducible $2$-dimension representation of $S_{4}$.

\begin{figure}[h]
\centering
\captionsetup{width=.8\linewidth}
\includegraphics[width = 16cm]{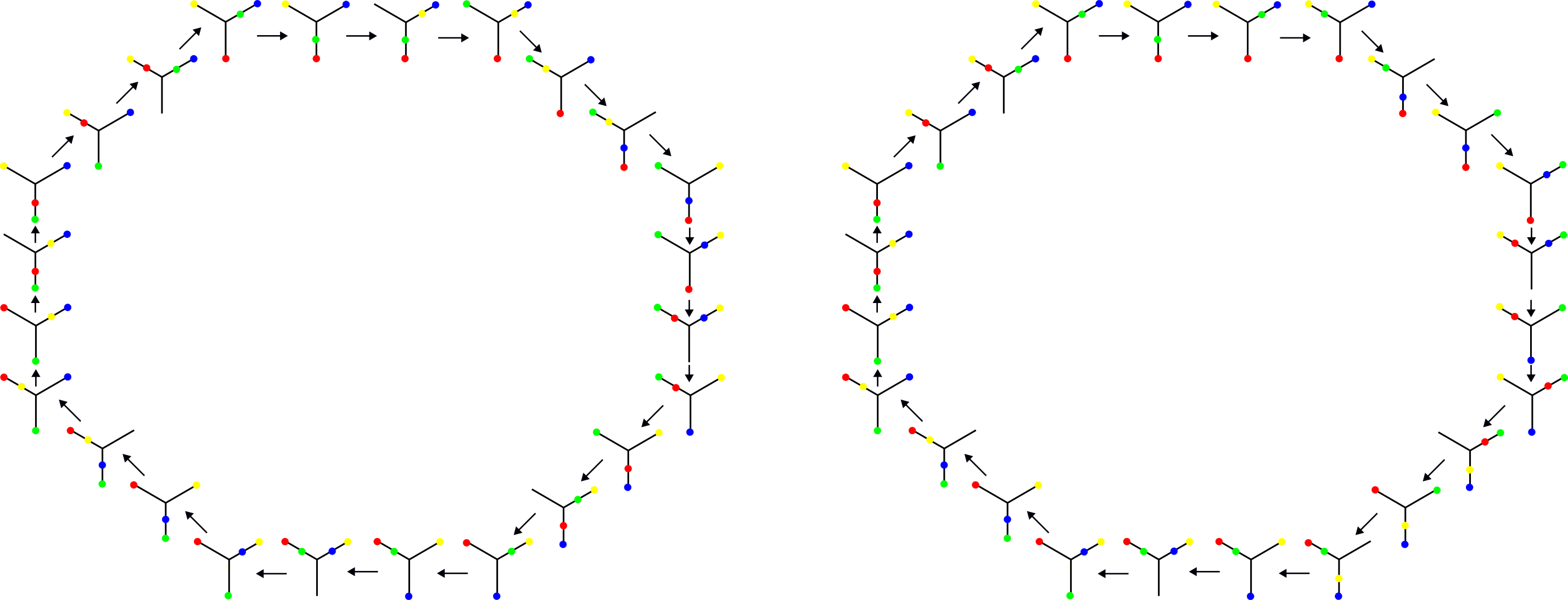}
\caption{Two classes in $H_{1}\big(F_{4}(\Gamma_{3})\big)$ generating $E^{2}_{1,0}[\Gamma_{3}](4)$.
}
\label{NewCyclesF4G3}
\end{figure}

\subsection{$\Gamma_{4}$}

In this subsection we provide pictures of the $E^{2}_{1, 0}[\Gamma_{k}]$-generators of $H_{1}\big(F_{\bullet}(\Gamma_{4})\big)$.
As the proof of Theorem \ref{rep stability for star} showed, there are generators in degrees $2$ and $3$.
For $k\ge 5$, the degree $2$-generators for $H_{1}\big(F_{\bullet}(\Gamma_{4})\big)$ yield generators for $H_{1}\big(F_{\bullet}(\Gamma_{k})\big)$ arising from the inclusions of $\Gamma_{4}\hookrightarrow \Gamma_{k}$.
This is also true in the unordered case.
However, in that case, the degree $2$-generators for $H_{1}\big(C_{\bullet}(\Gamma_{4})\big)$ arise from the degree $2$-generators for $H_{1}\big(C_{\bullet}(\Gamma_{3})\big)$, something that does not hold in the ordered case.

\subsubsection{$n=2$}

When $n=2$, we have that $E^{2}_{1, 0}[\Gamma_{4}](2)\cong H_{1}\big(F_{2}(\Gamma_{4})\big)\cong \Z^{5}$.
In Figure \ref{NewCyclesF2G4} we demonstrate one generator for this space.
Note that in this generator the blue particle only moves along two of the edges of $\Gamma_{4}$ and the red particle only moves along the other two.
There are $\binom{4}{2}=6$ ways to choose pairs of edges, corresponding to $6$ classes in $H_{1}\big(F_{2}(\Gamma_{4})\big)$; any $5$ of these classes generate $E^{2}_{1, 0}[\Gamma_{4}](2)$, with a signed sum of all $6$ of them being trivial.
By \cite[Theorem 3.3]{lu14}, the rational homology of $F_{2}(\Gamma_{4})$ decomposes as 2 copies of the alternating representation and 3 copies of the trivial representation.
Additionally, there is a relation between the $60$ images in $H_{1}\big(F_{3}(\Gamma_{4})\big)$ of any $5$ of these generators.
This relation corresponds to the nonzero $E^{2}_{2,0}[\Gamma_{4}](3)$-entry of the Mayer--Vietoris spectral sequence.

\begin{figure}[h]
\centering
\captionsetup{width=.8\linewidth}
\includegraphics[width = 4cm]{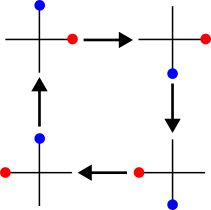}
\caption{A class in $H_{1}\big(F_{2}(\Gamma_{4})\big)$.
There are six such classes corresponding to how we pick a pair of edges for the blue particle to travel on.
}
\label{NewCyclesF2G4}
\end{figure}

\subsubsection{$n=3$}

When $n=3$, we have that $E^{2}_{1,0}[\Gamma_{4}](3)\cong \Z^{2}$, and in Figure \ref{NewCyclesF3G4} we demonstrate two generators for these classes.
After taking rational coefficients and quotienting by the classes that come from our insertion maps, these generators span two copies of the alternating representation of $S_{3}$.
Interchanging the roles of the edges of $\Gamma$ and the particles yields a duality between these classes and those of Figure \ref{NewCyclesF4G3}.

\begin{figure}[h]
\centering
\captionsetup{width=.8\linewidth}
\includegraphics[width = 12cm]{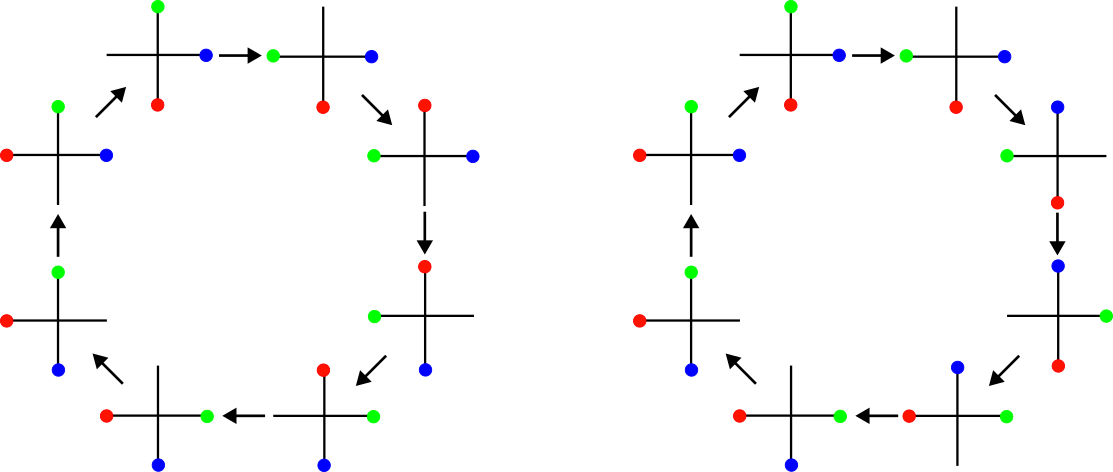}
\caption{Two classes in $H_{1}\big(F_{3}(\Gamma_{4})\big)$ generating $E^{2}_{1,0}[\Gamma_{4}](3)$.
}
\label{NewCyclesF3G4}
\end{figure}

\bibliographystyle{amsalpha}
\bibliography{RepStabForStarGraphs}

\end{document}